\documentclass[11pt,twoside]{amsart}
\usepackage{amsmath,amssymb,amsthm}
\usepackage{color}
\usepackage{cite}

\newtheorem{thm}{Theorem}[section]

 \newtheorem{cor}{Corollary}[section]
 \newtheorem{lem}{Lemma}[section]
 \newtheorem{prop}{Proposition}[section]
 \newtheorem{defn}{Definition}[section]
\newtheorem{rem}{Remark}[section]

\def\Id{{\rm Id}\,}

\def\tilde{\widetilde}

\newcommand\Z{\mathbb{Z}}

\renewcommand{\div}{\mbox{\rm div}\;\!}

\begin{document}
\title[compressible Navier-Stokes-Poisson system]{Optimal time-decay estimates for the
compressible Navier-Stokes-Poisson equations without additional smallness assumptions}

\author{Weixuan Shi}
\address{Department of Mathematics, Nanjing
University of Aeronautics and Astronautics,
Nanjing 211106, P.R.China,}
\email{wxshi168@163.com}

\subjclass{35Q35,35B40,76N15}
\keywords{Compressible Navier-Stokes-Poisson equations; time decay rates; $L^{p}$ critical spaces}

\begin{abstract}
The present paper is dedicated to the large time asymptotic behavior of global strong solutions near constant equilibrium (away from vacuum) to the compressible Navier-Stokes-Poisson equations. Precisely, we present that under the same regularity assumptions as in \cite{SX2}, a \textit{different} time-decay framework of the $\dot{B}_{p,1}^{s}$ norm of the critical global solutions is established. The proof mainly depends on the pure energy argument \textit{without the spectral analysis}, which allows us to remove \textit{the usual smallness assumption of low frequencies of initial data}.
\end{abstract}

\maketitle

\section{Introduction} \setcounter{equation}{0}
The barotropic compressible Navier-Stokes-Poisson equations can be written as
\begin{equation} \label{Eq:1.1}
\left\{
\begin{array}{l}
\partial_{t}\varrho +\mathrm{div}\big( \varrho u\big) =0,\\ [1mm]
\partial_{t}( \varrho u) +\mathrm{div}\big( \varrho u\otimes u\big)+\nabla P\big(\varrho\big)=\mathrm{div}\big(2\mu \,D(u) +\lambda\,\mathrm{div}\,u\,\mathrm{Id}\big)- \varrho \nabla \psi,\\ [1mm]
-\Delta \psi=\varrho-\varrho_{\infty},
\end{array}
\right.
\end{equation}
which can be used to simulate the transport of charged particles in semiconductor devices under the influence of electric fields (see for example \cite{MRS} for more explanations). Here, $u=u(t,x)\in \mathbb{R}^{d}$ (with $(t,x)\in [0, +\infty)\times \mathbb{R}^{d}$ ) and $\varrho =\varrho (t,x)$ stand for the
velocity field of charged particles and the density, respectively. The function $\psi=\psi(t,x)$ denotes the electrostatic potential force. The barotropic assumption means that the pressure $P$ is given suitably smooth function of $\varrho$. The notation $D(u)\triangleq\frac{1}{2}(\nabla u+{}^T\!\nabla u)$ stands for  the {\it deformation tensor}, and $\nabla$ and $\div$ are the gradient and divergence operators with respect to the space variable. The Lam\'{e} coefficients $\lambda$ and $\mu$ (the \textit{bulk} and \textit{shear viscosities}) are density-dependent functions, which are supposed to be smooth functions of density and to satisfy
\begin{equation}\label{Eq:1.2}
\mu>0 \ \ \hbox{and} \ \ \lambda +2\mu>0.
\end{equation}
The initial condition of \eqref{Eq:1.1} is prescribed by
\begin{equation}\label{Eq:1.3}
(\varrho,u)|_{t=0}=(\varrho _{0}(x),u_{0}(x)), \ \  x\in \mathbb{R}^{d}.
\end{equation}
We focus on solutions that are close to some constant equilibrium $(\varrho _{\infty},0)$ with $\varrho _{\infty}>0$, at infinity.

As for the many systems arising from mathematical physics, it is well known that  \textit{scaling invariance} plays a fundamental role. The main aim of this paper is to investigate the asymptotic behavior of (strong) global solutions to Cauchy problem \eqref{Eq:1.1}-\eqref{Eq:1.3} in the critical Besov spaces, that is in functional spaces endowed with norms that is invariant for all $l>0$ by the following transformation
\begin{equation*}
\varrho (t,x) \rightsquigarrow \varrho \big( l^{2}t,l x\big),\ \ u(t,x) \rightsquigarrow l u\big( l^{2}t,l x\big) ,\ \ \psi(t,x) \rightsquigarrow l^{-2}\psi\big(l^{2}t,l x\big).
\end{equation*}
Indeed, that definition of criticality corresponds to the scaling invariance  of system \eqref{Eq:1.1} if neglecting the lower order pressure term and electronic field term. To the best of our knowledge, the idea is now classical and comes from the research of incompressible Navier-Stokes system, see \cite{CM,FK,KY} and references therein. It is worth mentioning that if there isn't the Poisson potential, then system \eqref{Eq:1.1} reduces to the usual compressible Navier-Stokes system for baratropic fluids. There are some known results for compressible Navier-Stokes equations in the critical Besov spaces, see \cite{CD1,CMZ,DR1,DR2,DX,HB,OM,XJ,XX} and references therein. As regards the global existence for \eqref{Eq:1.1} in the critical Besov spaces, Hao and Li \cite{HL} established the global existence of small strong solution to \eqref{Eq:1.1} in the $L^{2}$ critical hybrid Besov space (in dimension $d\geq3$). Subsequently, Zheng \cite{ZXX} extended the result of \cite{HL} to the $L^{p}$ framework. However, they only considered the non oscillation case $2\leq p < d$ with $d\geq3$. In \cite{CO}, Chikami and Ogawa performed Lagrangian approach and then used Banach fixed point theorem to establish the local existence and uniqueness of solutions, which allows one to handle dimension $d\geq2$ in the general $L^{p}$ ($1<p<2d$) critical Besov spaces. Recently, Chikami and Danchin \cite{CD2} constructed the unique global strong solution near constant equilibrium in the critical $L^{p}$ framework and in any dimension $d\geq2$, where the cases of $2 \leq p \leq d$ and $p>d$ are both involved. For simplicity, those physical coefficients $\mu$ and $\lambda$ are assumed to be constant. In fact, their results still hold true in case that $\mu$ and $\lambda$ depend smoothly on the density. For convenience, we state it as follows (the reader is also referred to \cite{CD2}).
\begin{thm}\label{Thm1.1} Let $d\geq2$ and $p$ fulfill
\begin{equation}\label{Eq:1.4}
2\leq p\leq \min\big(4,2d/(d-2)\big) \ \hbox{and}, \ \hbox{additionally}, \ p\neq 4 \ \hbox{if} \ d=2.
\end{equation}
Assume that $P'\big(\varrho_{\infty} \big)>0$ and that \eqref{Eq:1.2} is satisfied.
There exists a small positive constant $c=c\big(p,d,\mu,\lambda, P,\varrho_{\infty}\big)$ and a universal integer $j_{0}\in
\mathbb{N}$ on $\mu$ and $\lambda$ such that if $a_{0}\triangleq (\varrho_{0}-\varrho_{\infty}) \in \dot{B}_{p,1}^{\frac {d}{p}}$, if $u_{0}\in \dot{B}_{p,1}^{\frac {d}{p}-1}$ and if in addition $(a_{0}^{\ell},\nabla u_{0}^{\ell} )\in \dot{B}_{2,1}^{\frac {d}{2}-2}$ with
\begin{equation*}
\mathcal{E}_{p,0}\triangleq \|( a_{0},\nabla u_{0}) \|_{\dot{B}_{2,1}^{\frac {d}{2}-2}}^{\ell}
+\|(\nabla a_{0},u_{0})\|_{\dot{B}_{p,1}^{\frac
{d}{p}-1}}^{h}\leq c,
\end{equation*}
then the Cauchy problem \eqref{Eq:1.1}-\eqref{Eq:1.3} admits a unique global-in-time solution $(\varrho,u)$ with $\varrho=\varrho_{\infty}+a$ and $(a,u)$ in the space $X_{p}$ defined by:
\begin{eqnarray*}
&&a^{\ell} \in \tilde{\mathcal{C}_{b}}(\mathbb{R_{+}};\dot{B}_{2,1}^{\frac {d}{2}-2})\cap L^{1}(\mathbb{R_{+}};\dot{B}_{2,1}^{\frac {d}{2}}), \ \ \ u^{\ell} \in \tilde{\mathcal{C}_{b}}(\mathbb{R_{+}};\dot{B}_{2,1}^{\frac {d}{2}-1})\cap L^{1}(\mathbb{R_{+}};\dot{B}_{2,1}^{\frac {d}{2}+1}), \\
&&a^{h}\in \tilde{\mathcal{C}_{b}}(\mathbb{R_{+}};\dot{B}_{p,1}^{\frac{d}{p}})\cap L^{1}(\mathbb{R_{+}};\dot{B}_{p,1}^{\frac {d}{p}}), \ \ \ \ u^{h}\in \tilde{\mathcal{C}_{b}}(\mathbb{R_{+}};\dot{B}_{p,1}^{\frac {d}{p}-1})\cap L^{1}(\mathbb{R_{+}};\dot{B}_{p,1}^{\frac {d}{p}+1}).
\end{eqnarray*}
Furthermore, we get for some constant $C=C(p,d,\mu,\lambda,P,\varrho_{\infty})$,
\begin{equation*}
\mathcal{E}_{p}(t)\leq C\mathcal{E}_{p,0},
\end{equation*}
for any $t>0$, where
\begin{equation*}
\mathcal{E}_{p}(t)\triangleq\|(a,\nabla u)\|_{\widetilde{L}_{t}^{\infty} (\dot{B}_{2,1}^{\frac {d}{2}-2})}^{\ell}+\|(a,\nabla u)\|_{L_{t}^{1} (\dot{B}_{2,1}^{\frac {d}{2}})}^{\ell}
+\|(\nabla a,u)\|_{\widetilde{L}_{t}^{\infty}(\dot{B}_{p,1}^{\frac {d}{p}-1})}^{h}
+\|(a,\nabla u)\|_{L_{t}^{1}(\dot{B}_{p,1}^{\frac {d}{p}})}^{h}.
\end{equation*}
\end{thm}
Next, a natural question is to explore the large time asymptotic description of solutions in Theorem \ref{Thm1.1} in the general $L^{p}$ critical Besov spaces. Here, let us recall the spectral analysis briefly, which has been studied by Li, Matsumura \& Zhang \cite{LMZ}. By the detailed analysis of the Fourier transform of the Green function for the linearized system of \eqref{Eq:1.1}, it follows from \cite{LMZ} that if the initial perturbation $\left(\varrho_{0}-\varrho_{\infty}, u_{0}\right) \in L^{q}(\mathbb{R}^{3}) \cap H^{s}(\mathbb{R}^{3})$ with $q\in [1,2]$ and $s\geq 3$, the optimal time decay rates of $L^{q}$-$L^{2}$ type is
\begin{equation}\label{Eq:1.8}
\|(\varrho-\varrho_{\infty})(t)\|_{L^{2}}\leq C (1+t)^{-\frac{3}{2}(\frac{1}{q}-\frac{1}{2})} \ \ \hbox{and} \ \
\|u(t)\|_{L^{2}}\leq C (1+t)^{-\frac{3}{2}(\frac{1}{q}-\frac{1}{2})+\frac{1}{2}}.
\end{equation}
It is observed that the electric field has significant effects on the large time behavior of the density and velocity field, which is a different ingrdient in comparison with the situation of compressible Navier-Stokes equations. Indeed, we see that the $L^{2}$ norm of the velocity $u$ grows in time at the rate $\langle t\rangle^{\frac{1}{2}}$ possibly if taking $q=2$ in \eqref{Eq:1.8}, which seems to be a contradiction with the dissipativity of \eqref{Eq:1.1}. Note that the energy structure of \eqref{Eq:1.1}, it is natural to assume that $\nabla\psi_{0}\in L^{2}$. 
With the aid of the Poisson equation, we find that the condition $\nabla\psi_{0}\in L^{2}$ is equivalent to that $\Lambda^{-1}(\varrho_{0}-\varrho_{\infty})\in L^{2}$ with $\Lambda^{-1}f\triangleq\mathcal{F}^{-1}(|\xi|^{-1}\mathcal{F}f)$. Inspired by this, Wang \cite{WYJ} posed a stronger assumption, say $(\Lambda^{-1}(\varrho_{0}-\varrho_{\infty}),u_{0})\in L^{q}(\mathbb{R}^{3})$ with $q\in [1,2]$, which leads to the substituting decay in comparison with \eqref{Eq:1.8}:
\begin{equation*}
\|(\varrho-\varrho_{\infty})(t)\|_{L^{2}}\leq C (1+t)^{-\frac{3}{2}(\frac{1}{q}-\frac{1}{2})-\frac{1}{2}} \ \ \hbox{and} \ \
\|u(t)\|_{L^{2}}\leq C(1+t)^{-\frac{3}{2}(\frac{1}{q}-\frac{1}{2})}.
\end{equation*}
In this sense, the electric field does not slow down but rather enhances the dissipation of density such that it enjoys additional half time-decay rate than velocity in time, with the \textit{relatively} stronger assumption than \cite{LMZ}. So far there are lots of works dedicated to the convergence rates of solutions with high Sobolev regularity to \eqref{Eq:1.1}-\eqref{Eq:1.3}, see \cite{LMZ,LZ, WYJ,WW1,WW2} and references therein. 

Let us also take a look at some important progress concerning \eqref{Eq:1.1}-\eqref{Eq:1.3} in the critical framework. Bie, Wang \& Yao \cite{BWY} developed the method of \cite{DX} so as to get the large-time asymptotic behavior of the constructed solutions in \cite{ZXX}. However, owing to the global-in-time results, they only consider the non oscillation case $2\leq p < d$ and $d\geq3$. Recently, Chikami and Danchin \cite{CD2} proposed the description of the time-decay which allows one to handle dimension $d\geq2$ in the $L^{2}$ critical Besov space. In \cite{SX2}, the author \& Xu developed a new regularity assumption of low frequencies, where the regularity $s_{1}$ belongs to $(1-\frac{d}{2},s_{0}]$ with $s_{0}\triangleq \frac{2d}{p}-\frac{d}{2}$, and established the \textit{sharp} time-weighted inequality, which led to the optimal time-decay rates of strong solutions. These recent works (see for example  \cite{BWY,CD2,LMZ,LZ,WW1,WW2,SX2} and references therein) mainly relies on the refined time-weighted energy approach in \textit{the Fourier semi-group framework}, so the smallness assumption of low frequencies of initial data plays a key role. In this paper, we develop a pure energy method of \cite{XJ} without the spectral analysis under the same regularity assumption as in \cite{SX2}, which enables one to remove the smallness of low frequencies of initial data, and then establish the optimal time-decay rates of solutions to \eqref{Eq:1.1}-\eqref{Eq:1.3} in the general $L^{p}$ critical Besov spaces. It is convenient to rewrite \eqref{Eq:1.1} as the nonlinear perturbation form of $(\varrho_{\infty},0)$, looking at the nonlinearities as source terms. To make a clearer introduction to our result, we assume that $\varrho_{\infty}=1$ and $P'(1)=1$. Consequently, in terms of the new variables $(a,u)$, system \eqref{Eq:1.1} becomes
\begin{equation} \label{Eq:1.9}
\left\{
\begin{array}{l}
\partial _{t}a+ \mathrm{div} u=f, \\ [1mm]
\partial _{t}u-\mathcal{A}u+\nabla a+\nabla(-\Delta)^{-1} a=g,
\end{array}
\right.
\end{equation}
where
\begin{eqnarray*}
&&f \triangleq-\div \big(au\big), \\
&&g \triangleq-u\cdot \nabla u-I(a)\mathcal{A}u-k(a)\nabla a+\frac {1}{1+a}\mathrm{div}\big(2\widetilde{\mu}(a)\,D(u)
+\widetilde{\lambda}(a) \mathrm{div} u \,\Id \big)
\end{eqnarray*}
with
\begin{eqnarray*}
&\mathcal{A}\triangleq\mu _{\infty} \Delta +\big(\lambda _{\infty} +\mu _{\infty}\big)
\nabla \div \ \ \hbox{such that} \ \ 2\mu_{\infty}+\lambda _{\infty}=1 \ \ \hbox{and} \ \
\mu_{\infty} >0 \\
&(\mu _{\infty}\triangleq\mu (1) \ \ \hbox{and} \ \ \lambda _{\infty} \triangleq\lambda (1)), \ \ I(a) \triangleq \frac {a}{1+a}, \ \
k(a) \triangleq \frac{P^{\prime}(1+a)}{1+a}-1, \\
&\tilde{\mu }(a)\triangleq \mu (1+a)-\mu (1), \ \ \ \ \tilde{\lambda}(a)\triangleq \lambda(1+a)-\lambda (1).
\end{eqnarray*}
Denote $\Lambda^{s}f\triangleq\mathcal{F}^{-1}(|\xi|^{s}\mathcal{F}f)$ for $s\in \mathbb{R}$. Now, we state main results as follows.
\begin{thm}\label{Thm1.2}
Let those assumptions of Theorem \ref{Thm1.1} hold and $(\varrho,u)$ be the
corresponding global solution to \eqref{Eq:1.1}. If in addition $a^{\ell}_{0}\in \dot{B}_{2,\infty}^{-s_{1}-1}$ and $u^{\ell}_{0}\in \dot{B}_{2,\infty}^{-s_{1}}$ ($1-\frac{d}{2}<s_{1}\leq s_{0} \ \ \hbox{with} \ \ s_{0}\triangleq \frac{2d}{p}-\frac{d}{2}$) such that $\|a_{0}\|^{\ell}_{\dot{B}_{2,\infty}^{-s_{1}-1}}$ and $\|u_{0}\|^{\ell}_{\dot{B}_{2,\infty}^{-s_{1}}}$ are bounded, then we have
\begin{eqnarray*}
\|(\varrho-1)(t)\|_{\dot{B}_{p,1}^{s}}&\lesssim&
(1+t)^{-\frac{d}{2}(\frac{1}{2}-\frac{1}{p})-\frac{s_{1}+s+1}{2}} \ \ \hbox{if} \ \ -\tilde{s}_{1}-1<s\leq\frac{d}{p}-2,\\
\|u(t)\|_{\dot{B}_{p,1}^{s}}&\lesssim& (1+t)^{-\frac{d}{2}(\frac{1}{2}-\frac{1}{p})-\frac{s_{1}+s}{2}} \ \ \ \ \ \ \ \ \hbox{if} \ \ -\tilde{s}_{1}<s\leq\frac{d}{p}-1,
\end{eqnarray*}
for all $t\geq0$, where $\tilde{s}_{1}\triangleq s_{1}+d\big(\frac{1}{2}-\frac{1}{p}\big)$.
\end{thm}
Moreover, one has the density decay estimates of $\dot{B}^{-s_{1}-1}_{2,\infty}$-$L^{r}$ type and the velocity decay estimates of $\dot{B}^{-s_{1}}_{2,\infty}$-$L^{r}$ type.
\begin{cor}\label{Cor1.1}
	Let those assumptions of Theorem \ref{Thm1.2} be satisfied. Then the corresponding solution $\varrho$ fulfills
	\begin{equation*}
	\|\Lambda^{l}(\varrho-1)(t)\|_{L^{r}}\lesssim
	(1+t)^{-\frac{d}{2}(\frac{1}{2}-\frac{1}{r})-\frac{s_{1}+l+1}{2}}
	\end{equation*}
	for $p\leq r\leq\infty$ and $l\in\mathbb{R}$ satisfying $-\tilde{s}_{1}-1<l+d\big(\frac{1}{p}-\frac{1}{r}\big)\leq\frac{d}{p}-2$, and $u$ fulfills
	$$\|\Lambda^{m}u(t)\|_{L^{r}}\lesssim (1+t)^{-\frac{d}{2}(\frac{1}{2}-\frac{1}{r})-\frac{s_{1}+m}{2}}$$
	for $p\leq r\leq\infty$ and $m\in\mathbb{R}$ satisfying $-\tilde{s}_{1}<m+d\big(\frac{1}{p}-\frac{1}{r}\big)\leq\frac{d}{p}-1$.
\end{cor}
\begin{rem}\label{Rem1.1}
In comparison with the recent work \cite{BWY,SX2,CD2}, the innovative ingredient is that the smallness of low frequencies is no longer needed in Theorem \ref{Thm1.2} and Corollary \ref{Cor1.1}. Moreover, noting that condition \eqref{Eq:1.4} allows us to consider the case
$p>d$ for which the velocity regularity exponent $\frac{d}{p}-1$ may be negative in physical dimensions $d=2,3$. The result thus applies to large highly oscillating initial velocity. Owing to the dissipative effect coming from the Poisson potential, we see that the $\dot{B}^{s}_{p,1}$ norm of density is faster at the rate $(1+t)^{-\frac{1}{2}}$ than that of velocity, which is an essential and different ingredient
in comparison with compressible Navier-Stokes equations (see for example \cite{DX,XJ,XX}).
\end{rem}
\begin{rem}\label{Rem1.2}
In our work \cite{SX2}, there is a little loss on time-decay rates due to using different Sobolev embeddings at low frequencies and high frequencies. For example, in the case of $s_{1}=s_{0}$, it was shown that the density (the velocity) itself decayed to equilibrium in $L^{p}$ norm with the rate of $(1+t)^{-d\big(\frac{1}{p}-\frac{1}{4}\big)-\frac{1}{2}}$ ($(1+t)^{-d\big(\frac{1}{p}-\frac{1}{4}\big)}$) for $t\rightarrow \infty$. The present results avoid this minor flaw and indicate the optimal decay of the density (the velocity) as fast as $(1+t)^{-\frac{d}{2p}-\frac{1}{2}}$ ($(1+t)^{-\frac{d}{2p}}$), which are satisfactory.
\end{rem}
Let us end this section by sketching the strategy for the proof of Theorem \ref{Thm1.2}. Throughout the process, the main task of this paper is to establish a Lyapunov-type inequality in time for energy norms (see \eqref{Eq:5.5}) by means of the pure energy method of \cite{XX} (independent of spectral analysis). Indeed, in the Sobolev framework of high Sobolev regularity, the idea was initiated by Strain \& Guo \cite{SG} for several Boltzmann type equations and then developed by Guo \& Wang\cite{GW} for compressible Navier-Stokes equations and Wang \cite{WYJ} for compressible Navier-Stokes-Poisson equations. In the critical regularity framework however, there is a loss of one derivative of density (see the term $u\cdot \nabla a$) in the transport equation. Clearly, their method fails to take effect in critical spaces.

Due to the coupled Poisson potential, there is a nonlocal term $\nabla(-\Delta)^{-1}a$ (that is equivalent to $\tilde{a}=\Lambda^{-1}a$) available in the velocity equation. By applying the hyperbolic energy approach, one can get the parabolic decay for the low frequencies of $(\tilde{a},u)$. In the high-frequency regime, the nonlocal term $\nabla(-\Delta)^{-1}a$ is no longer effective and the large behaves of solutions the same as that of compressible Navier-Stoke system. Consequently, one can get the dissipative mechanism of \eqref{Eq:1.1}-\eqref{Eq:1.2}. On the other hand, owing to the general regularity that $s_{1}$ belongs to the whole range $(1-\frac{d}{2},s_{0}]$ with $s_{0}=\frac{2d}{p}-\frac{d}{2}$, the low-frequency analysis is more complicated. Precisely, in light of low and high frequency decomposition, one splits the nonlinear term $\big(\Lambda^{-1}f,g\big)$ into $\big(\Lambda^{-1}f^{\ell},g^{\ell}\big)$ and $\big(\Lambda^{-1}f^{h},g^{h}\big)$ (see section 4). To bound the nonlinear term $\big(\Lambda^{-1}f^{\ell},g^{\ell}\big)$, we develop some non classical Besov product estimates \eqref{Eq:4.4}-\eqref{Eq:4.5} to get desired result. For the term $\big(\Lambda^{-1}f^{h},g^{h}\big)$, we proceed differently the analysis depending on whether $2\leq p\leq d$ (non oscillation) and $p>d$ (oscillation). The former case depends on Besov product estimates (see \eqref{Eq:4.6}), while the later case (that is relevant in physical dimension $d=2,3$) lies in non-classical product estimates in Proposition \ref{Prop2.5}. Combining these estimates leads to the evolution of Besov norm of solutions. Finally, nonlinear product estimates and real interpolations allow us to obtain the Lyapunov-type inequality \eqref{Eq:5.5} for energy norms.

The rest of the paper unfolds as follows: In section 2, we briefly recall Littlewood-Paley decomposition, Besov spaces  and useful analysis tools. In section 3, we establish
the low-frequency and high-frequency estimates of solutions. Section 4 is devoted to bounding the evolution of negative Besov norms,
which plays the key role in deriving the Lyapunov-type inequality for energy norms. In
the last section (Section 5), we show the proofs of Theorem \ref{Thm1.2} and Corollary \ref{Cor1.1}.

\section{Preliminary}\setcounter{equation}{0}
Throughout the paper, $C>0$ stands for a generic ``constant''. For brevity, we write $f\lesssim g$ instead of $f\leq Cg$. The notation $f\approx g$ means that $f\lesssim g$ and $g\lesssim f$. For any Banach space $X$ and $f,g\in X$, we agree that $\|(f,g)\| _{X}\triangleq \|f\| _{X}+\|g\|_{X}$. For all $T>0$ and $\theta \in[1,+\infty]$, we denote by $L_{T}^{\theta}(X) \triangleq L^{\theta}([0,T];X)$ the set of measurable functions $f:[0,T]\rightarrow X$ such that $t\mapsto\|f(t)\|_{X}$ is in $L^{\theta}(0,T)$.

\subsection{Littlewood-Paley decomposition and Besov spaces} \setcounter{equation}{0}
Let us recall Littlewood-Paley decomposition and Besov spaces for convenience. The reader is referred to Chap. 2 and Chap. 3 of \cite{BCD} for more details. Choose a smooth radial non increasing function $\chi $ with $\mathrm{Supp}\,\chi \subset
B\big(0,\frac {4}{3}\big)$ and $\chi \equiv 1$ on $B\big(0,\frac{3}{4}\big)$. Set $\varphi (\xi) =\chi (\xi/2)-\chi (\xi)$. It is not difficult to check that
\begin{equation*}
\sum_{j\in \mathbb{Z}}\varphi ( 2^{-j}\cdot ) =1\ \ \hbox{in}\ \
\mathbb{R}^{d}\setminus \{ 0\} \ \ \hbox{and}\ \ \mathrm{Supp}\,\varphi \subset \big\{ \xi \in \mathbb{R}^{d}:3/4\leq |\xi|\leq 8/3\big\}.
\end{equation*}
The homogeneous dyadic blocks $\dot{\Delta}_{j}$ ($j\in\mathbb{Z}$) are defined by
\begin{equation*}
\dot{\Delta}_{j}f\triangleq \varphi (2^{-j}D)f=\mathcal{F}^{-1}\big(\varphi
(2^{-j}\cdot )\mathcal{F}f\big)=2^{jd}h(2^{j}\cdot )\star f \ \ \hbox{with}\ \
h\triangleq \mathcal{F}^{-1}\varphi.
\end{equation*}
Consequently, one has the unit decomposition for any tempered distribution $f\in S^{\prime }(\mathbb{R}^{d})$
\begin{equation} \label{Eq:2.1}
f=\sum_{j\in \mathbb{Z}}\dot{\Delta}_{j}f.
\end{equation}
As it holds only modulo polynomials, it is convenient to consider the subspace of those tempered distributions $f$ such that
\begin{equation}\label{Eq:2.2}
\lim_{j\rightarrow -\infty }\| \dot{S}_{j}f\| _{L^{\infty}}=0,
\end{equation}
where $\dot{S}_{j}f$ stands for the low frequency cut-off defined by $\dot{S}_{j}f\triangleq\chi (2^{-j}D)f$. Indeed, if \eqref{Eq:2.2} is fulfilled, then \eqref{Eq:2.1} holds in $S'(\mathbb{R}^{d})$. For convenience, we denote by $S'_{0}(\mathbb{R}^{d})$ the subspace of tempered distributions satisfying \eqref{Eq:2.2}.

In terms with Littlewood-Paley decomposition, Besov spaces are defined as follows.
\begin{defn}\label{Defn2.1}
For $\sigma\in \mathbb{R}$ and $1\leq p,r\leq\infty,$ the homogeneous
Besov spaces $\dot{B}^{\sigma}_{p,r}$ is defined by
$$\dot{B}^{\sigma}_{p,r}\triangleq\big\{f\in S'_{0}:\|f\|_{\dot{B}^{\sigma}_{p,r}}<+\infty\big\},$$
where
\begin{equation}\label{Eq:2.3}
\|f\|_{\dot B^{\sigma}_{p,r}}\triangleq\|(2^{j\sigma}\|\dot{\Delta}_{j}  f\|_{L^p})\|_{\ell^{r}(\Z)}.
\end{equation}
\end{defn}
On the other hand, a class of mixed space-time Besov spaces are also used, which was initiated by J.-Y. Chemin and N. Lerner \cite{CL} (see also \cite{CJY} for the particular case of Sobolev spaces).
\begin{defn}\label{Defn2.2}
For $T>0, \sigma\in\mathbb{R}, 1\leq r,\theta\leq\infty$, the homogeneous Chemin-Lerner space $\widetilde{L}^{\theta}_{T}(\dot{B}^{\sigma}_{p,r})$
is defined by
\begin{equation*}
\widetilde{L}^{\theta}_{T}(\dot{B}^{\sigma}_{p,r})\triangleq\big\{f\in L^{\theta}(0,T;S'_{0}):\|f\|_{\widetilde{L}^{\theta}_{T}(\dot{B}^{\sigma}_{p,r})}<+\infty\big\},
\end{equation*}
where
\begin{equation}\label{Eq:2.4}
\|f\|_{\widetilde{L}^{\theta}_{T}(\dot{B}^{\sigma}_{p,r})}\triangleq\|(2^{j\sigma}\|\dot{\Delta}_{j}f\|_{L^{\theta}_{T}(L^{p})})\|_{\ell^{r}(\mathbb{Z})}.
\end{equation}
\end{defn}
For notational simplicity, index $T$ is omitted if $T=+\infty $.
We agree with the notation
\begin{equation*}
\tilde{\mathcal{C}}_{b}(\mathbb{R_{+}};\dot{B}_{p,r}^{\sigma})\triangleq \big\{f \in
\mathcal{C}(\mathbb{R_{+}};\dot{B}_{p,r}^{\sigma})\ \hbox{s.t}\ \big\|f\big\| _{\tilde{L}^{\infty}(\dot{B}_{p,r}^{\sigma})}<+\infty \big\} .
\end{equation*}
The Chemin-Lerner space $\tilde{L}^{\theta}_{T}(\dot{B}^{\sigma}_{p,r})$ may be linked with the standard spaces $L_{T}^{\theta} (\dot{B}_{p,r}^{\sigma})$ by means of Minkowski's
inequality.
\begin{rem}\label{Rem2.1}
It holds that
$$\|f\|_{\widetilde{L}^{\theta}_{T}(\dot{B}^{\sigma}_{p,r})}\leq\|f\|_{L^{\theta}_{T}(\dot{B}^{\sigma}_{p,r})}\,\,\,
\mbox{if} \,\, \, r\geq\theta;\ \ \ \
\|f\|_{\widetilde{L}^{\theta}_{T}(\dot{B}^{\sigma}_{p,r})}\geq\|f\|_{L^{\theta}_{T}(\dot{B}^{\sigma}_{p,r})}\,\,\,
\mbox{if}\,\,\, r\leq\theta.
$$
\end{rem}
Restricting the above norms \eqref{Eq:2.3} and \eqref{Eq:2.4} to the low or high frequencies parts of distributions will be fundamental in our method. For instance, let us fix some integer $j_{0}$ (the value of which will
follow from the proof of the high-frequency estimates) and put\footnote{Note that for technical reasons, we need a small
overlap between low and high frequencies.}
$$\| f\| _{\dot{B}_{p,1}^{\sigma}}^{\ell} \triangleq \sum_{j\leq
j_{0}}2^{j\sigma}\| \dot{\Delta}_{j}f\|_{L^{p}} \ \mbox{and} \ \|f\|_{\dot{B}_{p,1}^{\sigma}}^{h}\triangleq \sum_{j\geq j_{0}-1}2^{j\sigma}\| \dot{\Delta}_{j}f\| _{L^{p}},$$
\begin{equation*}
\|f\| _{\tilde{L}_{T}^{\infty} (\dot{B}_{p,1}^{\sigma})}^{\ell} \triangleq \sum_{j\leq j_{0}}2^{j\sigma}\|\dot{\Delta}_{j}f\|_{L_{T}^{\infty} (L^{p})} \
\mbox{and} \ \|f\| _{\tilde{L}_{T}^{\infty} (\dot{B}_{p,1}^{\sigma})}^{h}\triangleq \sum_{j\geq j_{0}-1}2^{j\sigma}\| \dot{\Delta}_{j}f\|
_{L_{T}^{\infty} (L^{p})}.
\end{equation*}
\subsection{Analysis tools in Besov spaces}
Let us recall the classical properties (see \cite{BCD}):
\begin{prop} \label{Prop2.1}
\begin{itemize}
\item \ \emph{Scaling invariance:} For any $\sigma\in \mathbb{R}$ and $(p,r)\in
[1,\infty ]^{2}$, there exists a constant $C=C(\sigma,p,r,d)$ such that for all $\lambda >0$ and $f\in \dot{B}_{p,r}^{\sigma}$, we have
\begin{equation*}
C^{-1}\lambda ^{\sigma-\frac {d}{p}}\|f\|_{\dot{B}_{p,r}^{\sigma}}
\leq \big\|f(\lambda \cdot)\big\|_{\dot{B}_{p,r}^{\sigma}}\leq C\lambda ^{\sigma-\frac {d}{p}}\|f\|_{\dot{B}_{p,r}^{\sigma}}.
\end{equation*}

\item \emph{Completeness:} $\dot{B}^{\sigma}_{p,r}$ is a Banach space whenever $
\sigma<\frac{d}{p}$ or $\sigma\leq \frac{d}{p}$ and $r=1$.

\item \emph{Action of Fourier multipliers:} If $F$ is a smooth homogeneous of
degree $m$ function on $\mathbb{R}^{d}\backslash \{0\}$ then
\begin{equation*}
F(D):\dot{B}_{p,r}^{\sigma}\rightarrow \dot{B}_{p,r}^{\sigma-m}.
\end{equation*}
\end{itemize}
\end{prop}

\begin{prop} \label{Prop2.2} Let $1\leq p,r_{1},r_{2}, r\leq \infty$.
\begin{itemize}
\item \emph{Complex interpolation:} If $f\in \dot{B}_{p,r_{1}}^{\sigma_{1}} \cap \dot{B}_{p,r_2}^{\sigma_{2}}$ and  $\sigma_{1}\neq \sigma_{2}$, then $f\in \dot{B}_{p,r}^{\theta \sigma_{1}+(1-\theta )\sigma_{2}}$ for
all $\theta \in (0,1)$ and
\begin{equation*}
\|f\|_{\dot{B}_{p,r}^{\theta \sigma_{1}+(1-\theta )\sigma_{2}}}\lesssim \|f\| _{\dot{B}_{p,r_{1}}^{\sigma_{1}}}^{\theta}
\|f\|_{\dot{B}_{p,r_2}^{\sigma_{2}}}^{1-\theta }
\end{equation*}
with $\frac{1}{r}=\frac{\theta}{r_{1}}+\frac{1-\theta}{r_{2}}$.

\item \emph{Real interpolation:} If $f\in \dot{B}_{p,r_{1}}^{\sigma_{1}} \cap \dot{B}_{p,r_2}^{\sigma_{2}}$ and  $\sigma_{1}< \sigma_{2}$, then $f\in \dot{B}_{p,r}^{\theta \sigma_{1}+(1-\theta )\sigma_{2}}$ for all $\theta \in (0,1)$ and
\begin{equation*}
\|f\|_{\dot{B}_{p,r}^{\theta \sigma_{1}+(1-\theta )\sigma_{2}}}\lesssim \frac{C}{\theta(1-\theta)(\sigma_{2}-\sigma_{1})}\|f\| _{\dot{B}_{p,\infty}^{\sigma_{1}}}^{\theta} \|f\|_{\dot{B}_{p,\infty}^{\sigma_{2}}}^{1-\theta }.
\end{equation*}
\end{itemize}
\end{prop}
The following embedding properties will be used frequently throughout this paper.
\begin{prop} \label{Prop2.3} (Embedding for Besov spaces on $\mathbb{R}^{d}$)
\begin{itemize}
\item For any $p\in[1,\infty]$ we have the  continuous embedding
$\dot {B}^{0}_{p,1}\hookrightarrow L^{p}\hookrightarrow \dot {B}^{0}_{p,\infty}.$
\item If $\sigma\in\mathbb{R}$, $1\leq p_{1}\leq p_{2}\leq\infty$ and $1\leq r_{1}\leq r_{2}\leq\infty,$
then $\dot {B}^{\sigma}_{p_1,r_1}\hookrightarrow
\dot {B}^{\sigma-d\,(\frac{1}{p_{1}}-\frac{1}{p_{2}})}_{p_{2},r_{2}}$.
\item The space $\dot {B}^{\frac {d}{p}}_{p,1}$ is continuously embedded in the set  of
bounded  continuous functions (going to zero at infinity if, additionally, $p<\infty$).
\end{itemize}
\end{prop}
The product estimate in Besov spaces plays a fundamental role in bounding bilinear terms of \eqref{Eq:1.9} (see for example \cite{BCD,DX,SX1,SX2,XJ}).
\begin{prop}\label{Prop2.4}
Let $\sigma>0$ and $1\leq p,r\leq\infty$. Then $\dot{B}^{\sigma}_{p,r}\cap L^{\infty}$ is an algebra and
\begin{equation*}
\|fg\|_{\dot{B}^{\sigma}_{p,r}}\lesssim \|f\|_{L^{\infty}}\|g\|_{\dot{B}^{\sigma}_{p,r}}+\|g\|_{L^{\infty}}\|f\|_{\dot{B}^{\sigma}_{p,r}}.
\end{equation*}
Let the real numbers $\sigma_{1},$ $\sigma_{2},$ $p_1$  and $p_2$ fulfill
\begin{equation*}
\sigma_{1}+\sigma_{2}>0,\quad \sigma_{1}\leq\frac {d}{p_{1}},\quad\sigma_{2}\leq\frac {d}{p_{2}},\quad
\sigma_{1}\geq\sigma_{2},\quad\frac{1}{p_{1}}+\frac{1}{p_{2}}\leq1.
\end{equation*}
Then we have
\begin{equation*}
\|fg\|_{\dot{B}^{\sigma_{2}}_{q,1}}\lesssim \|f\|_{\dot{B}^{\sigma_{1}}_{p_{1},1}}\|g\|_{\dot{B}^{\sigma_{2}}_{p_{2},1}}\quad\hbox{with}\quad
\frac1{q}=\frac1{p_{1}}+\frac1{p_{2}}-\frac{\sigma_{1}}d.
\end{equation*}
Additionally, for exponents $\sigma>0$ and $1\leq p_{1},p_{2},q\leq\infty$ satisfying
\begin{equation*}
\frac{d}{p_{1}}+\frac{d}{p_{2}}-d\leq \sigma \leq\min\left(\frac {d}{p_{1}},\frac {d}{p_{2}}\right)\quad\hbox{and}\quad \frac{1}{q}=\frac {1}{p_{1}}+\frac {1}{p_{2}}-\frac{\sigma}{d},
\end{equation*}
we have
\begin{equation*}
\|fg\|_{\dot{B}^{-\sigma}_{q,\infty}}\lesssim\|f\|_{\dot{B}^{\sigma}_{p_{1},1}}\|g\|_{\dot{B}^{-\sigma}_{p_{2},\infty}}.
\end{equation*}
\end{prop}
Proposition \ref{Prop2.4} is not enough to bound the low frequency part of some nonlinear terms in the proof of Theorem
\ref{Thm1.2}, so we need to the following non-classical product estimate (see \cite{SX1,SX2,DX,XJ}).
\begin{prop}\label{Prop2.5} Let $j_{0}\in\Z,$ and denote $z^{\ell}\triangleq\dot S_{j_{0}}z$, $z^{h}\triangleq z-z^{\ell}$ and, for any $\sigma\in\mathbb{R}$,
\begin{equation*}
\|z\|_{\dot B^{\sigma}_{2,\infty}}^{\ell}\triangleq\sup_{j\leq j_{0}}2^{j\sigma} \|\dot{\Delta}_{j}z\|_{L^2}.
\end{equation*}
There exists a universal integer $N_{0}$ suc that  for any $2\leq p\leq 4$ and $\sigma>0$, we have
\begin{eqnarray}\label{Eq:2.7}
&&\|f g^{h}\|_{\dot {B}^{-s_{0}}_{2,\infty}}^{\ell}\leq C \big(\|f\|_{\dot {B}^{\sigma}_{p,1}}+\|\dot S_{j_{0}+N_{0}}f\|_{L^{{p}^{*}}}\big)\|g^{h}\|_{\dot{B}^{-\sigma}_{p,\infty}},\\
\label{Eq:2.8}
&&\|f^{h} g\|_{\dot {B}^{-s_{0}}_{2,\infty}}^{\ell}
\leq C \big(\|f^{h}\|_{\dot{B}^{\sigma}_{p,1}}+\|\dot{S}_{j_{0}+N_{0}}f^{h}\|_{L^{p^{*}}}\big)\|g\|_{\dot {B}^{-\sigma}_{p,\infty}}
\end{eqnarray}
with  $s_{0}\triangleq \frac{2d}{p}-\frac {d}{2}$ and $\frac1{p^{*}}\triangleq\frac{1}{2}-\frac{1}{p},$
and $C$ depending only on $j_{0}$, $d$ and $\sigma$.
\end{prop}
System \eqref{Eq:1.9} also involves compositions of functions (through $I(a)$, $k(a)$, $\tilde{\lambda}(a)$ and $\tilde{\mu}(a)$) that
are handled according to the following conclusion.
\begin{prop}\label{Prop2.6}
Let $F:\mathbb{R}\rightarrow\mathbb{R}$ be  smooth with $F(0)=0$.
For all  $1\leq p,r\leq\infty$ and $\sigma>0$ we have
$F(f)\in \dot {B}^{\sigma}_{p,r}\cap L^{\infty}$  for  $f\in \dot {B}^{\sigma}_{p,r}\cap L^{\infty},$  and
\begin{equation*}
\|F(f)\|_{\dot B^\sigma_{p,r}}\leq C\|f\|_{\dot B^\sigma_{p,r}}
\end{equation*}
with $C$ depending only on $\|f\|_{L^{\infty}}$, $F'$ \emph{(}and higher derivatives\emph{)}, $\sigma$, $p$ and $d$.

In the case $\sigma>-\min\big(\frac {d}{p},\frac {d}{p'}\big)$ then $f\in\dot {B}^{\sigma}_{p,r}\cap\dot {B}^{\frac {d}{p}}_{p,1}$
implies that $F(f)\in \dot {B}^{\sigma}_{p,r}\cap\dot {B}^{\frac {d}{p}}_{p,1}$, and we have
$$\|F(f)\|_{\dot B^{\sigma}_{p,r}}\leq C(1+\|f\|_{\dot {B}^{\frac {d}{p}}_{p,1}})\|f\|_{\dot {B}^{\sigma}_{p,r}}.$$
\end{prop}
%*****************************

\section{Low-frequency and high-frequency analysis} \setcounter{equation}{0}
Let us establish a Lyapunov-type inequality for energy norms by means of a pure energy approach. For clarity, the proof is divided into two steps. In this section, we focus on the low-frequency and high-frequency estimates.
\subsection{Low-frequency estimates}
Let $\Lambda^{s}z\triangleq\mathcal{F}^{-1}\big(|\xi|^{s}\mathcal{F}z\big)$ ($s\in\mathbb{R}$). Denote by $\Omega=\Lambda^{-1}\mathrm{curl}\,u$ the incompressible part of $u$ and by $\omega=\Lambda^{-1}\mathrm{div}\,u$ the compressible part of $u$. Then system \eqref{Eq:1.9} becomes
\begin{equation*}
\left\{
\begin{array}{l}
\partial _{t}a +\Lambda \omega=f, \\ [1mm]
\partial _{t}\omega-\Delta \omega-\Lambda a-\Lambda^{-1}a=h, \\ [1mm]
\partial _{t}\Omega-\mu_{\infty}\Delta \Omega=m, \\ [1mm]
u=-\Lambda^{-1}\nabla\omega+\Lambda^{-1}\mathrm{div}\,\Omega
\end{array}
\right.
\end{equation*}
with $h=\Lambda^{-1}\mathrm{div}\,g$ and $m=\Lambda^{-1}\mathrm{curl}\,g$.

Observe that the incompressible component $\Omega$ satisfies
\begin{equation*}
\partial _{t}\Omega-\mu_{\infty}\Delta \Omega=m.
\end{equation*}
It is not difficult to infer that
\begin{equation}\label{Eq:3.1}
\frac{d}{dt}\|\Omega^{\ell}\|_{\dot{B}^{\frac{d}{p}-1}_{p,1}}+\|\Omega^{\ell}\|_{\dot{B}^{\frac{d}{p}+1}_{p,1}}\lesssim \|g^{\ell}\|_{\dot{B}^{\frac{d}{p}-1}_{p,1}}
\end{equation}
for $1\leq p \leq \infty$ and $t\geq 0$, where $z^{\ell}\triangleq \dot{S}_{j_{0}}z$.
As pointed out in Introduction, at low frequencies, it
is natural to consider the following system
\begin{equation}\label{Eq:3.2}
\left\{
\begin{array}{l}
\partial _{t}\tilde{a}+\omega=\Lambda^{-1}f, \\ [1mm]
\partial _{t}\omega-\Delta \omega-(\Lambda^{2}+1) \tilde{a}=h
\end{array}
\right.
\end{equation}
with $\tilde{a}=\Lambda^{-1}a$.
\begin{lem}\label{Lem3.1}
let $j_{0}$ be some integer. Then it holds that for all $t\geq 0$
\begin{equation}\label{Eq:3.3}
\frac{d}{dt}\|(\tilde{a},u)^{\ell}\|_{\dot{B}^{\frac{d}{2}-1}_{2,1}}+\|(\tilde{a},u)^{\ell}\|_{\dot{B}^{\frac{d}{2}+1}_{2,1}}\lesssim \|(\Lambda^{-1}f^{\ell},g^{\ell})\|_{\dot{B}^{\frac{d}{2}-1}_{2,1}}.
\end{equation}
\end{lem}
\begin{proof}
Let $z_{j}\triangleq \dot{\Delta}_{j}z$. We may apply the operator $\dot{\Delta}_{j}\dot{S}_{j_{0}}$ to \eqref{Eq:3.2}. Taking advantage of the standard energy method, we get the following four equalities:
\begin{equation}\label{Eq:3.4}
\frac{1}{2}\frac{d}{dt}\big(\|\tilde{a}_{j}^{\ell}\|^{2}_{L^{2}}+\|\omega^{\ell}_{j}\|^{2}_{L^{2}}\big)+\|\Lambda \omega^{\ell}_{j}\|_{L^{2}}^{2}
=(\Lambda^{2}\tilde{a}^{\ell}_{j}|\omega^{\ell}_{j})+(\Lambda^{-1}f^{\ell}_{j}|\tilde{a}^{\ell}_{j})
+(h^{\ell}_{j}|\omega^{\ell}_{j}),
\end{equation}
\begin{eqnarray} \label{Eq:3.5}
-\frac{d}{dt}(\omega^{\ell}_{j}|\Lambda^{2}\tilde{a}^{\ell}_{j})+\|\Lambda^{2} \tilde{a}^{\ell}_{j}\|^{2}_{L^{2}}+\|\Lambda \tilde{a}^{\ell}_{j}\|^{2}_{L^{2}}=(\Lambda^{2}\omega^{\ell}_{j}|\Lambda^{2} \tilde{a}^{\ell}_{j})+\|\Lambda \omega^{\ell}_{j}\|^{2}_{L^{2}} \nonumber\\
-(\Lambda f^{\ell}_{j}|\omega^{\ell}_{j})-(h^{\ell}_{j}|\Lambda^{2}\tilde{a}^{\ell}_{j}),
\end{eqnarray}
\begin{equation}\label{Eq:3.6}
\frac{1}{2}\frac{d}{dt}\|\Lambda^{2}\tilde{a}^{\ell}_{j}\|^{2}_{L^{2}}=-(\Lambda^{2}\omega^{\ell}_{j}|\Lambda^{2} \tilde{a}^{\ell}_{j})+(\Lambda^{2} f^{\ell}_{j}|\Lambda\tilde{a}^{\ell}_{j})
\end{equation}
and
\begin{equation}\label{Eq:3.7}
\frac{1}{2}\frac{d}{dt}\|\Lambda \tilde{a}^{\ell}_{j}\|^{2}_{L^{2}}=-(\omega^{\ell}_{j}|\Lambda^{2} \tilde{a}^{\ell}_{j})+(\Lambda f^{\ell}_{j}|\tilde{a}^{\ell}_{j}).
\end{equation}
Combining \eqref{Eq:3.4}-\eqref{Eq:3.7}, one can conclude that
\begin{eqnarray*}
\frac{1}{2}\frac{d}{dt}\mathcal{L}_{j}^{2}+\|(\Lambda^{2} \tilde{a}^{\ell}_{j},\Lambda \tilde{a}^{\ell}_{j},\Lambda \omega^{\ell}_{j})\|^{2}_{L^{2}}=2(\Lambda^{-1}f^{\ell}_{j}|\tilde{a}^{\ell}_{j})+2(h^{\ell}_{j}|\omega^{\ell}_{j})+2(\Lambda f^{\ell}_{j}|\tilde{a}^{\ell}_{j})\\
-(\Lambda f^{\ell}_{j}|\omega^{\ell}_{j})
-(h^{\ell}_{j}|\Lambda^{2}\tilde{a}^{\ell}_{j})+(\Lambda^{2} f^{\ell}_{j}|\Lambda\tilde{a}^{\ell}_{j})
\end{eqnarray*}
with $\mathcal{L}_{j}^{2}\triangleq 2\big(\|\tilde{a}^{\ell}_{j}\|^{2}_{L^{2}}+\|\omega^{\ell}_{j}\|^{2}_{L^{2}}+\|\Lambda \tilde{a}^{\ell}_{j}\|^{2}_{L^{2}}\big)+\|\Lambda^{2}\tilde{a}^{\ell}_{j}\|^{2}_{L^{2}}-2(\omega^{\ell}_{j}|\Lambda^{2}\tilde{a}^{\ell}_{j})$. Thanks to the low-frequency cut-off, we get from Young inequality that $\mathcal{L}_{j}^{2}\thickapprox \|(\tilde{a}^{\ell}_{j},\Lambda \tilde{a}^{\ell}_{j},\Lambda^{2}\tilde{a}^{\ell}_{j},\omega^{\ell}_{j})\|^{2}_{L^{2}}\thickapprox \|(\Lambda\tilde{a}^{\ell}_{j},\tilde{a}^{\ell}_{j},\omega^{\ell})\|^{2}_{L^{2}}\thickapprox \|(\tilde{a}^{\ell}_{j},\omega^{\ell})\|^{2}_{L^{2}}$. Consequently, it follows that
\begin{equation*}
\frac{1}{2}\frac{d}{dt}\mathcal{L}_{j}^{2}+2^{2k}\mathcal{L}_{j}^{2}\lesssim\|(\Lambda^{-1}f^{\ell}_{j},\Lambda f^{\ell}_{j},\Lambda^{2}f^{\ell}_{j},g^{\ell}_{j})\|_{L^{2}}\mathcal{L}_{j},
\end{equation*}
which leads to
\begin{equation*}
\frac{d}{dt}\mathcal{L}_{j}+2^{2j}\mathcal{L}_{j}\lesssim \|(\Lambda^{-1}f^{\ell}_{j},g^{\ell}_{j})\|_{L^{2}}.
\end{equation*}
Therefore, multiplying both sides by $2^{j\big(\frac{d}{2}-1\big)}$ and summing up on $j\in \mathbb{Z}$ yield
\begin{equation}\label{Eq:3.8}
\frac{d}{dt}\|(\tilde{a},\omega)^{\ell}\|_{\dot{B}^{\frac{d}{2}-1}_{2,1}}+\|(\tilde{a},\omega)^{\ell}\|_{\dot{B}^{\frac{d}{2}+1}_{2,1}}\lesssim \|(\Lambda^{-1}f,g)^{\ell}\|_{\dot{B}^{\frac{d}{2}-1}_{2,1}}.
\end{equation}
Hence, \eqref{Eq:3.3} is followed by \eqref{Eq:3.1} and \eqref{Eq:3.8} directly.
\end{proof}
\subsection{High-frequency estimates}
In the high-frequency regime, the nonlocal term $\nabla(-\Delta)^{-1}a$ no longer effective and the Green function behaves the same as that of compressible Navier–Stokesequations. Consequently, we apply the $L^{p}$ energy method and then establish the high-frequency estimates (see \cite{XX} for more details).
\begin{lem}\label{Lem3.2}
let $j_{0}$ be chosen suitably large. It holds that for all $t\geq 0$
\begin{eqnarray}\label{Eq:3.10}
&&\frac{d}{dt}\|(\nabla a,u)\|_{\dot{B}_{p,1}^{\frac {d}{p}-1}}^{h}
+\big(\|\nabla a\|_{\dot{B}_{p,1}^{\frac{d}{p}-1}}^{h}+\|u\|_{\dot{B}_{p,1}^{\frac {d}{p}+1}}^{h}\big)\nonumber\\
&\lesssim & \|f\|^{h}_{\dot{B}^{\frac{d}{p}-2}_{p,1}}+\|g\|^{h}_{\dot{B}^{\frac{d}{p}-1}_{p,1}}
+\|\nabla u\|_{\dot{B}^{\frac{d}{p}}_{p,1}}\|a\|_{\dot{B}^{\frac{d}{p}}_{p,1}},
\end{eqnarray}
where $$\|z\|^{h}_{\dot{B}^{s}_{p,1}}\triangleq \sum _{j\geq j_{0}-1}2^{js}\|\dot{\Delta}_{j}z\|_{L^{p}} \ \ \hbox{for} \ \ s\in\mathbb{R}.$$
\end{lem}
%%%%%%%%%%%%%%%%%%%%%%%%%%%%%%%%%%%%%%%%%
\section{The evolution of negative Besov norm}  \setcounter{equation}{0}
This section is devoted to the evolution of Besov norms at low frequencies, which plays the key role in deriving the Lyapunov-type inequality for energy norms.
\begin{lem}\label{Lem4.1}
Let $1-\frac{d}{2}<s_{1}\leq s_{0}$ and $p$ satisfy \eqref{Eq:1.4}. It hold that:
\begin{eqnarray} \label{Eq:4.1}
\big(\|(\tilde{a},u)(t)\|^{\ell}_{\dot{B}^{-s_{1}}_{2,\infty}}\big)^{2}
&\lesssim& \big(\|(\tilde{a}_{0},u_{0})\|^{\ell}_{\dot{B}^{-s_{1}}_{2,\infty}}\big)^{2}
+\int_{0}^{t}D^{1}_{p}(\tau)\big(\|(\tilde{a},u)(\tau)\|^{\ell}_{\dot{B}^{-s_{1}}_{2,\infty}}\big)^{2}d\tau \nonumber\\
&&+\int_{0}^{t} D^{2}_{p}(\tau)\|(\tilde{a},u)(\tau)\|^{\ell}_{\dot{B}^{-s_{1}}_{2,\infty}}d\tau,
\end{eqnarray}
where
\begin{eqnarray*}
&&D_{p}^{1}\triangleq \|a\|_{\dot{B}^{\frac{d}{p}}_{p,1}}+\|u\|^{\ell}_{\dot{B}^{\frac{d}{2}+1}_{2,1}}
+\|u\|^{h}_{\dot{B}^{\frac{d}{p}+1}_{p,1}},\\
&&D_{p}^{2}\triangleq \big(\|a\|^{\ell}_{\dot{B}^{\frac{d}{2}-2}_{2,1}}+\|u\|^{\ell}_{\dot{B}^{\frac{d}{2}-1}_{2,1}}
+\|a\|^{h}_{\dot{B}^{\frac{d}{p}}_{p,1}}+\|u\|^{h}_{\dot{B}^{\frac{d}{p}-1}_{p,1}}\big)\big(\|a\|^{h}_{\dot{B}^{\frac{d}{p}}_{p,1}}
+\|u\|^{h}_{\dot{B}^{\frac{d}{p}+1}_{p,1}}\big).
\end{eqnarray*}
\end{lem}
\begin{proof}
First of all, let us keep in mind that the global solution $(a,u)$ given by Theorem \ref{Thm1.1} satisfies
\begin{equation} \label{Eq:4.2}
\|a\|_{\tilde{L}^{\infty}_{t}(\dot{B}^{\frac{d}{p}}_{p,1})}\leq c \ll 1 \ \ \hbox{for all} \ \ t\geq 0.
\end{equation}
It follows from \eqref{Eq:3.4} and \eqref{Eq:3.7} that
\begin{equation*}
\frac{1}{2}\frac{d}{dt}\|(\tilde{a}_{k},\Lambda \tilde{a}_{k},\omega_{k})\|^{2}_{L^{2}}+\|\Lambda \omega_{k}\|_{L^{2}}^{2}
\leq \|(\Lambda^{-1}f_{k},\Lambda f_{k},h_{k})\|_{L^{2}}\|(\tilde{a}_{k},\omega_{k})\|_{L^{2}}.
\end{equation*}
By performing a routine procedure, we can conclude that
\begin{equation}\label{Eq:4.3}
\big(\|(\tilde{a},u)(t)\|^{\ell}_{\dot{B}^{-s_{1}}_{2,\infty}}\big)^{2}\lesssim \big(\|(\tilde{a}_{0},u_{0})\|^{\ell}_{\dot{B}^{-s_{1}}_{2,\infty}}\big)^{2}
+\int_{0}^{t} \|(\Lambda^{-1}f, g)\|^{\ell}_{\dot{B}^{-s_{1}}_{2,\infty}}\|(\tilde{a},u)(\tau)\|^{\ell}_{\dot{B}^{-s_{1}}_{2,\infty}}d\tau.
\end{equation}
Next, let us handle the nonlinear norm $\|(\Lambda^{-1}f,g)\|^{\ell}_{\dot{B}^{-s_{1}}_{2,\infty}}$. To do this, it is suitable to decompose $\Lambda^{-1}f$ and $g$ according to low-frequency and high-frequency parts as follows:
$$\Lambda^{-1}f=\Lambda^{-1}f^{\ell}+\Lambda^{-1}f^{h}$$
with
\begin{equation*}
\Lambda^{-1}f^{\ell}\triangleq-\Lambda^{-1}\mathrm{div}\,(au^{\ell}), \ \ \ \ \ \Lambda^{-1}f^{h}\triangleq-\Lambda^{-1}\mathrm{div}\,(au^{h})
\end{equation*}
and
$$g= g^{\ell}+g^{h}$$
with
\begin{eqnarray*}
&&g^{\ell}\triangleq-u\cdot \nabla u^{\ell}-k( a) \nabla a^{\ell}+g_{3}(a,u^{\ell})+g_{4}(a,u^{\ell}),\\
&&g^{h}\triangleq-u\cdot \nabla u^{h}-k(a) \nabla a^{h}+g_{3}(a,u^{h})+g_{4}(a,u^{h}),
\end{eqnarray*}
where
\begin{eqnarray*}
g_{3}(a,v)&=&\frac {1}{1+a}\big( 2\widetilde{\mu }(a)\,\mathrm{div}\,D(v)+
\widetilde{\lambda}(a)\,\nabla \mathrm{div}\,v\big) -I(a)\mathcal{A}v, \\
g_{4}(a,v)&=&\frac {1}{1+a}\big(2\widetilde{\mu}'(a)\,\mathrm{div}\,D(v)\cdot\nabla a+\widetilde{\lambda}'(a)\,\mathrm{div}\,v\, \nabla
a\big)
\end{eqnarray*}
and
$$z^{\ell}\triangleq\sum_{j<j_{0}} \dot{\Delta}_{j}z,\ \  \ \ z^{h}\triangleq z-z^{\ell} \ \ \hbox{for} \ \ z=a,u.$$
As shown by \cite{XX}, we can get the following two estimates
\begin{eqnarray*}
\|u \cdot\nabla u^{\ell}\|^{\ell}_{\dot{B}^{-s_{1}}_{2,\infty}} \lesssim \big(\|u\|^{\ell}_{\dot{B}^{\frac{d}{2}+1}_{2,1}}+\| u\|^{h}_{\dot{B}^{\frac{d}{p}+1}_{p,1}}\big)\|u\|^{\ell}_{\dot{B}^{-s_{1}}_{2,\infty}},\\
\|u\cdot\nabla u^{h}\|^{\ell}_{\dot{B}^{-s_{1}}_{2,\infty}}
\lesssim \big(\|u\|^{\ell}_{\dot{B}^{\frac{d}{2}-1}_{2,1}}+\|u\|^{h}_{\dot{B}^{\frac{d}{p}-1}_{p,1}}\big)
\|u\|^{h}_{\dot{B}^{\frac{d}{p}+1}_{p,1}}.
\end{eqnarray*}
In order to finish the proof of lemma \ref{Lem4.1}, it suffices to bound those “new”
nonlinear terms, which are different terms in comparison with compressible Navier-Stokes equations.
For those terms of $\Lambda^{-1}f$ and $g$ with $a^{\ell}$ or $u^{\ell}$, we shall use the following two inequalities
\begin{eqnarray}\label{Eq:4.4}
\|FG\|_{\dot{B}^{-s_{1}}_{2,\infty}}&\lesssim& \|F\|_{\dot{B}^{\frac{d}{p}}_{p,1}}\|G\|_{\dot{B}^{-s_{1}}_{2,\infty}},\\
\label{Eq:4.5}
\|FG\|_{\dot{B}^{\frac{d}{p}-\frac{d}{2}-s_{1}}_{2,\infty}} &\lesssim& \|F\|_{\dot{B}^{\frac{d}{p}-1}_{p,1}}\|G\|_{\dot{B}^{\frac{d}{p}-\frac{d}{2}-s_{1}+1}_{2,\infty}}
\end{eqnarray}
for $1-\frac{d}{2}<s_{1}\leq s_{0}$. Indeed, the proofs of \eqref{Eq:4.4}-\eqref{Eq:4.5} may be found in \cite{XX}.
For the term with $\Lambda^{-1}\mathrm{div} (au^{\ell})$, we note that, owing to \eqref{Eq:4.4},
\begin{equation*}
\|\Lambda^{-1}\mathrm{div}\,(au^{\ell})\|^{\ell}_{\dot{B}^{-s_{1}}_{2,\infty}} \lesssim \|au^{\ell}\|^{\ell}_{\dot{B}^{-s_{1}}_{2,\infty}} \lesssim \|a\|_{\dot{B}^{\frac{d}{p}}_{p,1}}\|u^{\ell}\|_{\dot{B}^{-s_{1}}_{2,\infty}}
\lesssim \|a\|_{\dot{B}^{\frac{d}{p}}_{p,1}}\|u\|^{\ell}_{\dot{B}^{-s_{1}}_{2,\infty}}.
\end{equation*}
where $\Lambda^{-1}\mathrm{div}$ is an homogeneous Fourier multiplier of degree $0$.
Let us next look at the term involving $k(a)\nabla a^{\ell}$. Thanks to \eqref{Eq:4.4} and Proposition \ref{Prop2.6}, we arrive at
\begin{equation*}
\|k(a)\nabla a^{\ell}\|^{\ell}_{\dot{B}^{-s_{1}}_{2,\infty}} \lesssim \|k(a)\|_{\dot{B}^{\frac{d}{p}}_{p,1}}\|\nabla a^{\ell}\|_{\dot{B}^{-s_{1}}_{2,\infty}}
\lesssim \|a\|_{\dot{B}^{\frac{d}{p}}_{p,1}}\|\tilde{a}\|^{\ell}_{\dot{B}^{-s_{1}}_{2,\infty}}.
\end{equation*}
Similarly, we have
\begin{equation*}
\|g_{3}(a,u^{\ell})\|^{\ell}_{\dot{B}^{-s_{1}}_{2,\infty}}
\lesssim \|a\|_{\dot{B}^{\frac{d}{p}}_{p,1}}\|u\|^{\ell}_{\dot{B}^{-s_{1}}_{2,\infty}}.
\end{equation*}
To deal with the term with $g_{4}(a,u^{\ell})$, it follows from the fact $s_{1}\leq s_{1}+\frac{d}{2}-\frac{d}{p}$, \eqref{Eq:4.5} and Proposition \ref{Prop2.6} that
\begin{equation*}
\|g_{4}(a,u^{\ell})\|^{\ell}_{\dot{B}^{-s_{1}}_{2,\infty}} \lesssim
\|g_{4}(a,u^{\ell})\|^{\ell}_{\dot{B}^{\frac{d}{p}-\frac{d}{2}-s_{1}+1}_{2,\infty}}
\lesssim \|a\|_{\dot{B}^{\frac{d}{p}}_{p,1}}\|\nabla u^{\ell}\|_{\dot{B}^{\frac{d}{p}-\frac{d}{2}-s_{1}+1}_{2,\infty}}
\lesssim \|a\|_{\dot{B}^{\frac{d}{p}}_{p,1}}\| u\|^{\ell}_{\dot{B}^{-s_{1}}_{2,\infty}}.
\end{equation*}

For those ``new'' terms in $\Lambda^{-1}f^{h}$ and $g^{h}$, precisely,
$$\Lambda^{-1}\mathrm{div}\,(au^{h}), \ \ k(a)\nabla a^{h}, \ \ g_{3}(a,u^{h}) \ \ \hbox{and} \ \ g_{4}(a,u^{h}),$$
we shall proceed these calculations differently depending on whether $2\leq p\leq d$ or $p>d$. For the case $2\leq p\leq d$, we shall take advantage of the following inequality
\begin{equation}\label{Eq:4.6}
\|FG^{h}\|^{\ell}_{\dot{B}^{-s_{1}}_{2,\infty}}\lesssim \|FG^{h}\|^{\ell}_{\dot{B}^{-s_{0}}_{2,\infty}}
\lesssim \|F\|_{\dot{B}^{\frac{d}{p}-1}_{p,1}}
\|G^{h}\|_{\dot{B}^{\frac{d}{p}-1}_{p,1}} \ \ \hbox{for all} \ \ 1-\frac{d}{2}<s_{1}\leq s_{0},
\end{equation}
which has been shown by \cite{SX2}.
With the aid of \eqref{Eq:4.6}, we get 
\begin{equation*}
\|\Lambda^{-1}\mathrm{div}\,(au^{h})\|^{\ell}_{\dot{B}^{-s_{1}}_{2,\infty}}
\lesssim \|a\|_{\dot{B}^{\frac{d}{p}-1}_{p,1}}\|u^{h}\|_{\dot{B}^{\frac{d}{p}-1}_{p,1}}
\lesssim \big(\|a\|^{\ell}_{\dot{B}^{\frac{d}{2}-2}_{2,1}}+\|a\|^{h}_{\dot{B}^{\frac{d}{p}}_{p,1}}\big)
\|u\|^{h}_{\dot{B}^{\frac{d}{p}+1}_{p,1}}.
\end{equation*}
To bound the term corresponding to $k(a)\nabla a^{h}$, we can use \eqref{Eq:4.6}, \eqref{Eq:4.2} and Proposition \ref{Prop2.6} with $\frac{d}{p}-1>-\frac{d}{p}$ $(p<d<2d)$, and get
\begin{equation*}
\|k(a)\nabla a^{h}\|^{\ell}_{\dot{B}^{-s_{1}}_{2,\infty}}
\lesssim\|a\|_{\dot{B}^{\frac{d}{p}-1}_{p,1}}\|\nabla a^{h}\|_{\dot{B}^{\frac{d}{p}-1}_{p,1}}
\lesssim\big(\|a\|^{\ell}_{\dot{B}^{\frac{d}{2}-2}_{2,1}}+\|a\|^{h}_{\dot{B}^{\frac{d}{p}}_{p,1}}\big)\|a\|^{h}_{\dot{B}^{\frac{d}{p}}_{p,1}}.
\end{equation*}
Similarly, we arrive at
\begin{equation*}
\|g_{3}(a,u^{h})\|^{\ell}_{\dot{B}^{-s_{1}}_{2,\infty}}\lesssim \|a\|_{\dot{B}^{\frac{d}{p}-1}_{p,1}}\|\nabla ^{2}u^{h}\|_{\dot{B}^{\frac{d}{p}-1}_{p,1}}\lesssim \big(\|a\|^{\ell}_{\dot{B}^{\frac{d}{2}-2}_{2,1}}+\|a\|^{h}_{\dot{B}^{\frac{d}{p}}_{p,1}}\big)
\|u\|^{h}_{\dot{B}^{\frac{d}{p}+1}_{p,1}}.
\end{equation*}
For the term containing $g_{4}(a,u^{h})$, we write that, due to \eqref{Eq:4.6} and Proposition \ref{Prop2.6},
\begin{equation*}
\|g_{4}(a,u^{h})\|^{\ell}_{\dot{B}^{-s_{1}}_{2,\infty}} \lesssim\|a\|_{\dot{B}^{\frac{d}{p}}_{p,1}}\| u^{h}\|_{\dot{B}^{\frac{d}{p}}_{p,1}}
\lesssim \big(\|a\|^{\ell}_{\dot{B}^{\frac{d}{2}-2}_{2,1}}+\|a\|^{h}_{\dot{B}^{\frac{d}{p}}_{p,1}}\big)
\| u\|^{h}_{\dot{B}^{\frac{d}{p}+1}_{p,1}}.
\end{equation*}

In what follows, we consider the oscillation case $p>d$. Using the fact that $s_{1}\leq s_{0}$ and
applying \eqref{Eq:2.7} with $\sigma=1-\frac{d}{p}>0$ yield
\begin{equation}\label{Eq:4.7}
\|F G^{h}\|^{\ell}_{\dot{B}^{-s_{1}}_{2,\infty}}\lesssim
\|F G^{h}\|^{\ell}_{\dot{B}^{-s_{0}}_{2,\infty}}\lesssim \big(\|F\|_{\dot{B}^{1-\frac{d}{p}}_{p,1}}+\|F^{\ell}\|_{L^{p^{*}}}\big)\|G^{h}\|_{\dot{B}^{\frac{d}{p}-1}_{p,1}} \ \hbox{with} \ \frac{1}{p^{*}}\triangleq\frac{1}{2}-\frac{1}{p}.
\end{equation}
Using the embeddings $\dot{B}^{\frac{d}{p}}_{2,1}\hookrightarrow L^{p^{*}}$, $\dot{B}^{\frac{d}{2}-1}_{2,1}\hookrightarrow\dot{B}^{\frac{d}{p}-1}_{p,1}$ and the relations $\frac{d}{2}-1\leq\frac{d}{p}$ and $\frac{d}{p}-1<1-\frac{d}{p}<\frac{d}{p}$ for $p>d$ satisfying \eqref{Eq:1.4}, we arrive at
\begin{equation}\label{Eq:4.8}
\|F G^{h}\|^{\ell}_{\dot{B}^{-s_{1}}_{2,\infty}}\lesssim
\big(\|F^{\ell}\|_{\dot{B}^{\frac{d}{2}-1}_{2,1}}+\|F^{h}\|_{\dot{B}^{\frac{d}{p}}_{p,1}}\big)\|G^{h}\|_{\dot{B}^{\frac{d}{p}-1}_{p,1}}.
\end{equation}
Taking $F=a$ and $G=u$ in \eqref{Eq:4.8} yields
\begin{equation*}
\|\Lambda^{-1}\mathrm{div}\,(au^{h})\|^{\ell}_{\dot{B}^{-s_{1}}_{2,\infty}}
\lesssim \big(\|a\|^{\ell}_{\dot{B}^{\frac{d}{2}-2}_{2,1}}+\|a\|^{h}_{\dot{B}^{\frac{d}{p}}_{p,1}}\big)
\|u\|^{h}_{\dot{B}^{\frac{d}{p}+1}_{p,1}}.
\end{equation*}
We observe that, using the composition inequality in Lebesgue spaces, the embeddings $\dot{B}^{\frac{d}{p}}_{2,1}\hookrightarrow L^{p^{*}}$ and $\dot{B}^{s_{0}}_{p,1}\hookrightarrow L^{p^{*}}$ and the relations $\frac{d}{p}>\frac{d}{2}-2$ and $s_{0}\leq \frac{d}{p}$, we obtain
\begin{equation*}
\|k(a)\|_{L^{p^{*}}}\lesssim\|a\|_{L^{p^{*}}}\lesssim\|a^{\ell}\|_{\dot{B}^{\frac{d}{p}}_{2,1}}+\|a^{h}\|_{\dot{B}^{s_{0}}_{p,1}}
\lesssim\|a\|^{\ell}_{\dot{B}^{\frac{d}{2}-2}_{2,1}}+\|a\|^{h}_{\dot{B}^{\frac{d}{p}}_{p,1}}.
\end{equation*}
Hence, taking advantage of \eqref{Eq:4.7}, Proposition \ref{Prop2.6} and applying the fact $\frac{d}{p}-2<\frac{d}{p}-1<0<1-\frac{d}{p}<\frac{d}{p}$ and the embedding $\dot{B}^{\frac{d}{2}-2}_{2,1}\hookrightarrow\dot{B}^{\frac{d}{p}-2}_{p,1}$ yield
\begin{eqnarray} \label{Eq:4.9}
\|k(a)\nabla a^{h}\|^{\ell}_{\dot{B}^{-s_{1}}_{2,\infty}}
&\lesssim& \big(\|a\|_{\dot{B}^{1-\frac{d}{p}}_{p,1}}
+\|a\|^{\ell}_{\dot{B}^{\frac{d}{2}-2}_{2,1}}+\|a\|^{h}_{\dot{B}^{\frac{d}{p}}_{p,1}}\big)
\|\nabla a^{h}\|_{\dot{B}^{\frac{d}{p}-1}_{p,1}}\nonumber\\
&\lesssim & \big(\|a\|^{\ell}_{\dot{B}^{\frac{d}{2}-2}_{2,1}}+\|a\|^{h}_{\dot{B}^{\frac{d}{p}}_{p,1}}\big)
\| a\|^{h}_{\dot{B}^{\frac{d}{p}}_{p,1}}.
\end{eqnarray}
For the term with $g_{3}(a,u^{h})$, we mimic the procedure leading to \eqref{Eq:4.9} and get
\begin{equation*}
\|g_{3}(a,u^{h})\|^{\ell}_{\dot{B}^{-s_{1}}_{2,\infty}}
\lesssim  \big(\|a\|^{\ell}_{\dot{B}^{\frac{d}{2}-2}_{2,1}}+\|a\|^{h}_{\dot{B}^{\frac{d}{p}}_{p,1}}\big)
\|u\|^{h}_{\dot{B}^{\frac{d}{p}+1}_{p,1}}.
\end{equation*}
Let us finally bound the term with $g_{4}(a,u^{h})$. Applying \eqref{Eq:2.8} with $\sigma=1-\frac{d}{p}$ yields for any smooth function $K$ vanishing at $0$,
\begin{equation*}
\|\nabla K(a)\otimes \nabla u^{h}\|^{\ell}_{\dot{B}^{-s_{0}}_{2,\infty}}
\lesssim \big(\|\nabla u^{h}\|_{\dot{B}^{1-\frac{d}{p}}_{p,1}}
+\sum_{j=j_{0}}^{j_{0}+N_{0}-1}\|\dot{ \Delta}_{j}\nabla u^{h}\|_{L^{p^{*}}}\big)\|\nabla K(a)\|_{\dot{B}^{\frac{d}{p}-1}_{p,1}}.
\end{equation*}
As $p^{*}\geq p$, it follows from Bernstein inequality that
$\|\dot{ \Delta}_{j}\nabla u^{h}\|_{L^{p^{*}}}\lesssim \|\dot{ \Delta}_{j}\nabla u^{h}\|_{L^{p}}$ for $j_{0}\leq j<j_{0}+N_{0}$.
Hence, with the aid of Proposition \ref{Prop2.6} and the relations $s_{1}\leq s_{0}$ and $1-\frac{d}{p}<\frac{d}{p}$, we get
\begin{equation*}
\|g_{4}(a,u^{h})\|^{\ell}_{\dot{B}^{-s_{1}}_{2,\infty}}
\lesssim \|a\|_{\dot{B}^{\frac{d}{p}}_{p,1}}\|\nabla u^{h}\|_{\dot{B}^{1-\frac{d}{p}}_{p,1}}
\lesssim \big(\|a\|^{\ell}_{\dot{B}^{\frac{d}{2}-2}_{2,1}}+\|a\|^{h}_{\dot{B}^{\frac{d}{p}}_{p,1}}\big)
\|u\|^{h}_{\dot{B}^{\frac{d}{p}+1}_{p,1}}.
\end{equation*}
By inserting above all estimates into \eqref{Eq:4.3} gives \eqref{Eq:4.1}.
\end{proof}
According to the definition of $\mathcal{E}_{p}(t)$ in Theorem \ref{Thm1.1}, we deduce that
\begin{equation} \label{Eq:4.10}
\int_{0}^{t}D^{1}_{p}(\tau)d\tau \leq \mathcal{E}_{p}\leq C \mathcal{E}_{p,0}
\end{equation}
and
\begin{equation} \label{Eq:4.11}
\int_{0}^{t}D^{2}_{p}(\tau)d\tau \leq \mathcal{E}^{2}_{p}\leq C \mathcal{E}_{p,0},
\end{equation}
since $ \mathcal{E}_{p,0}\ll 1$. Finally, combining \eqref{Eq:4.10} and \eqref{Eq:4.11}, one can make use of nonlinear generalisations of the Gronwall's inequality (see for example, Page 360 of \cite{MPF}) and obtain
\begin{equation}\label{Eq:4.12}
\|(\tilde{a},u)(t,\cdot)\|^{\ell}_{\dot{B}^{-s_{1}}_{2,\infty}}\leq C_{0} \ \ \hbox{for all} \ \ t\geq 0,
\end{equation}
where $C_{0}>0$ depends on the norms $\|a_{0}\|^{\ell}_{\dot{B}^{-s_{1}-1}_{2,\infty}}$ and $\|u_{0}\|^{\ell}_{\dot{B}^{-s_{1}}_{2,\infty}}$.
\section{Proofs of main results} \setcounter{equation}{0}
In this section, we present the proofs of Theorem \ref{Thm1.2} and Corollary \ref{Cor1.1}.  
\subsection{Proof of Theorem \ref{Thm1.2}}
We get from Lemmas \ref{Lem3.1} and \ref{Lem3.2} that
\begin{eqnarray}\label{Eq:5.1}
\frac{d}{dt}\big(\|(\tilde{a},u)^{\ell}\|_{\dot{B}^{\frac{d}{2}-1}_{2,1}}+\|(\nabla a,u)\|_{\dot{B}_{p,1}^{\frac {d}{p}-1}}^{h}\big)
+\big(\|(\tilde{a},u)^{\ell}\|_{\dot{B}^{\frac{d}{2}+1}_{2,1}}+\|\nabla a\|_{\dot{B}_{p,1}^{\frac {d}{p}-1}}^{h}+\|u\|_{\dot{B}_{p,1}^{\frac {d}{p}+1}}^{h}\big)\nonumber\\
\lesssim\|(\Lambda^{-1}f,g)\|^{\ell}_{\dot{B}^{\frac{d}{2}-1}_{2,1}}+\|f\|^{h}_{\dot{B}^{\frac{d}{p}-2}_{p,1}}+\|g\|^{h}_{\dot{B}^{\frac{d}{p}-1}_{p,1}}
+\|\nabla u\|_{\dot{B}^{\frac{d}{p}}_{p,1}}\|a\|_{\dot{B}^{\frac{d}{p}}_{p,1}},
\end{eqnarray}
where
\begin{equation*}
\|z\|^{\ell}_{\dot{B}^{s}_{2,1}}\triangleq \sum _{j\leq j_{0}}2^{js}\|\dot{\Delta}_{j}z\|_{L^{p}} \ \ \hbox{for} \ \ s\in\mathbb{R}.
\end{equation*}
Note that the fact $p\geq2$, the last term above can be bounded easily by
$\mathcal{E}_{p}(t)(\|u^{\ell}\|_{\dot{B}_{2,1}^{\frac {d}{2}+1}}+\|u\|_{\dot{B}_{p,1}^{\frac {d}{p}+1}}^{h})$. Next, it follows from Proposition \ref{Prop2.4} and Bernstein inequality that
\begin{equation*}
\|f\|^{h}_{\dot{B}^{\frac{d}{p}-2}_{p,1}}\lesssim \|au\|^{h}_{\dot{B}^{\frac{d}{p}-1}_{p,1}}\lesssim \|a\|_{\dot{B}^{\frac{d}{p}}_{p,1}} \|u\|_{\dot{B}^{\frac{d}{p}-1}_{p,1}}\lesssim \mathcal{E}_{p}(t)\big(\|\tilde{a}^{\ell}\|_{\dot{B}^{\frac{d}{2}+1}_{2,1}} +\|a\|^{h}_{\dot{B}^{\frac{d}{p}}_{p,1}}\big).
\end{equation*}
On the other hand, we observe that
\begin{equation*}
\|g\|^{h}_{\dot{B}^{\frac{d}{p}-1}_{p,1}}\lesssim \|u\|_{\dot{B}^{\frac{d}{p}-1}_{p,1}}\|\nabla u\|_{\dot{B}^{\frac{d}{p}}_{p,1}}+\|a\|_{\dot{B}^{\frac{d}{p}}_{p,1}}\|\nabla u\|_{\dot{B}^{\frac{d}{p}}_{p,1}}+\|a\|^{2}_{\dot{B}^{\frac{d}{p}}_{p,1}}.
\end{equation*}
Furthermore, we arrive at
\begin{equation*}
\|g\|^{h}_{\dot{B}^{\frac{d}{p}-1}_{p,1}}\lesssim \mathcal{E}_{p}(t)\big(\|(\tilde{a},u)^{\ell}\|_{\dot{B}^{\frac{d}{2}+1}_{2,1}}+\|a\|^{h}_{\dot{B}^{\frac{d}{p}}_{p,1}}
+\|u\|_{\dot{B}_{p,1}^{\frac {d}{p}+1}}^{h}\big).
\end{equation*}
So we are left with the proof of
\begin{equation}\label{Eq:5.2}
\|(\Lambda^{-1}f,g)\|^{\ell}_{\dot{B}^{\frac{d}{2}-1}_{2,1}} \lesssim \mathcal{E}_{p}(t)\big(\|(\tilde{a},u)^{\ell}\|_{\dot{B}^{\frac{d}{2}+1}_{2,1}}+\|a\|^{h}_{\dot{B}^{\frac{d}{p}}_{p,1}}
+\|u\|_{\dot{B}_{p,1}^{\frac {d}{p}+1}}^{h}\big).
\end{equation}
In order to prove \eqref{Eq:5.2}, we need the following two inequalities (see \cite{CD2, DR2}):
\begin{equation}\label{Eq:5.3}
\|T_{a}b\| _{\dot{B}_{2,1}^{\sigma-1+\frac {d}{2}-\frac {d}{p}}}
\lesssim \|a\|_{\dot{B}_{p,1}^{\frac {d}{p}-1}}\|b\| _{\dot{B}_{p,1}^{\sigma}}
\ \hbox{if} \ \sigma\in\mathbb{R}, \ d\geq 2 \ \hbox{and} \ 1\leq p\leq \min (4,d^{*}),
\end{equation}
\begin{equation}\label{Eq:5.4}
\| R(a,b)\| _{\dot{B}_{2,1}^{\sigma-1+\frac {d}{2}-\frac {d}{p}}}
\lesssim \|a\| _{\dot{B}_{p,1}^{\frac {d}{p}-1}}\| b\| _{\dot{B}_{p,1}^{\sigma}}
\ \hbox{if} \ \sigma>1-\min\big(\frac {d}{p},\frac {d}{p'}\big) \
\hbox{and} \ 1\leq p\leq 4,
\end{equation}
where $\frac{1}{p}+\frac{1}{p'}=1$ and $d^{*}=\frac{2d}{d-2}$.

Note that $\Lambda^{-1}\mathrm{div}$ is an homogeneous
Fourier multiplier of degree $0$, we write that
\begin{equation*}
\|\Lambda^{-1}f\|^{\ell}_{\dot{B}^{\frac{d}{2}-1}_{2,1}} \lesssim \|au\|^{\ell}_{\dot{B}^{\frac{d}{2}-1}_{2,1}}.
\end{equation*}
Taking advantage of Bony's para-product decomposition gives
$au=T_{a}u+R(a,u)+T_{u}a.$
According to \eqref{Eq:5.3} and \eqref{Eq:5.4} with $\sigma=\frac{d}{p}$, we have
\begin{equation*}
\|T_{a}u\|^{\ell}_{\dot{B}^{\frac{d}{2}-1}_{2,1}}+\|R(a,u)\|^{\ell}_{\dot{B}^{\frac{d}{2}-1}_{2,1}}
+\|T_{u}a\|^{\ell}_{\dot{B}^{\frac{d}{2}-1}_{2,1}}
\lesssim \|a\|_{\dot{B}_{p,1}^{\frac {d}{p}-1}}\|u\| _{\dot{B}_{p,1}^{\frac{d}{p}}}
+ \|u\|_{\dot{B}_{p,1}^{\frac {d}{p}-1}}\|a\| _{\dot{B}_{p,1}^{\frac{d}{p}}}.
\end{equation*}
With the aid of the interpolations, embeddings and Young inequality, we observe that the right-side norm of the above three inequalities can be bounded by
$$\mathcal{E}_{p}(t)\big(\|(\tilde{a},u)^{\ell}\|_{\dot{B}^{\frac{d}{2}+1}_{2,1}}+\|a\|^{h}_{\dot{B}^{\frac{d}{p}}_{p,1}}
+\|u\|_{\dot{B}_{p,1}^{\frac {d}{p}+1}}^{h}\big).$$
Hence, we conclude that
\begin{equation*}
\|\Lambda^{-1}f\|^{\ell}_{\dot{B}^{\frac{d}{2}-1}_{2,1}} \lesssim \mathcal{E}_{p}(t)\big(\|(\tilde{a},u)^{\ell}\|_{\dot{B}^{\frac{d}{2}+1}_{2,1}}+\|a\|^{h}_{\dot{B}^{\frac{d}{p}}_{p,1}}
+\|u\|_{\dot{B}_{p,1}^{\frac {d}{p}+1}}^{h}\big).
\end{equation*}
It follows from \cite{XX} that
\begin{equation*}
\|u\cdot\nabla u\|^{\ell}_{\dot{B}_{2,1}^{\frac{d}{2}-1}} \lesssim \mathcal{E}_{p}(t) \big(\|u^{\ell}\| _{\dot{B}_{2,1}^{\frac{d}{2}+1}}+\|u\|^{h}_{\dot{B}_{p,1}^{\frac{d}{p}+1}}\big).
\end{equation*}
For the term with $k(a)\nabla a$, we use the decomposition
\begin{equation*}
k(a)\nabla a=T_{\nabla a}k(a)+R(\nabla a,k(a))+T_{k(a)}\nabla a.
\end{equation*}
With the aid of \eqref{Eq:5.3} and \eqref{Eq:5.4} with $\sigma=\frac{d}{p}$ as well as Proposition \ref{Prop2.6}, we get
\begin{equation*}
\|T_{\nabla a}k(a)\|^{\ell}_{\dot{B}^{\frac{d}{2}-1}_{2,1}}
\lesssim \|a\|^{2}_{\dot{B}_{p,1}^{\frac {d}{p}}},\ \
\|R(\nabla a,k(a))\|^{\ell}_{\dot{B}^{\frac{d}{2}-1}_{2,1}}
\lesssim \|a\|^{2}_{\dot{B}_{p,1}^{\frac {d}{p}}}
\end{equation*}
and, owing to \eqref{Eq:5.3} with $\sigma=\frac{d}{p}-1$ and Proposition \ref{Prop2.6}, \eqref{Eq:4.2}
\begin{equation*}
\|T_{k(a)}\nabla a\|^{\ell}_{\dot{B}^{\frac{d}{2}-1}_{2,1}}
\lesssim \|T_{k(a)}\nabla a\|^{\ell}_{\dot{B}^{\frac{d}{2}-2}_{2,1}}
\lesssim \|a\|_{\dot{B}_{p,1}^{\frac {d}{p}-1}}\|a\| _{\dot{B}_{p,1}^{\frac{d}{p}}}.
\end{equation*}
Hence, we have
\begin{equation*}
\|k(a)\nabla a\|^{\ell}_{\dot{B}_{2,1}^{\frac{d}{2}-1}}\lesssim \mathcal{E}_{p}(t) \big(\|\tilde{a}^{\ell}\| _{\dot{B}_{2,1}^{\frac{d}{2}+1}}+\|a\|^{h} _{\dot{B}_{p,1}^{\frac{d}{p}}}\big).
\end{equation*}
The term $g_{3}(a,u)$ of $g$ is of the type $H(a)\nabla^{2} u$ with $H(a)=0$. Consequently, we write that
\begin{equation*}
H(a)\nabla^{2} u=T_{\nabla^{2} u}H(a)+R(\nabla^{2}u,H(a))+T_{H(a)}\nabla ^{2} u.
\end{equation*}
We conclude that exactly as the previous term $k(a)\nabla a$ that
\begin{equation*}
\|g_{3}(a,u)\|^{\ell}_{\dot{B}_{2,1}^{\frac{d}{2}-1}}\lesssim \mathcal{E}_{p}(t) \big(\|u^{\ell}\| _{\dot{B}_{2,1}^{\frac{d}{2}+1}}+\|u\|^{h} _{\dot{B}_{p,1}^{\frac{d}{p}+1}}\big).
\end{equation*}
The term $g_{4}(a,u)$ of $g$ is of the type $\nabla K(a)\otimes \nabla u$ with $K(0)=0$. Let us use the decomposition
\begin{equation*}
\nabla K(a)\otimes\nabla u=T_{\nabla K(a)}\nabla u+R(\nabla K(a),\nabla u)+T_{\nabla u}\nabla K(a).
\end{equation*}
By virtue of \eqref{Eq:5.3} and \eqref{Eq:5.4} with $\sigma=\frac{d}{p}$ and Proposition \ref{Prop2.6}, we have
\begin{equation*}
\|T_{\nabla K(a)}\nabla u\|^{\ell}_{\dot{B}^{\frac{d}{2}-1}_{2,1}}
+\|R(\nabla K(a),\nabla u)\|^{\ell}_{\dot{B}^{\frac{d}{2}-1}_{2,1}}
\lesssim\|a\|_{\dot{B}_{p,1}^{\frac {d}{p}}}\|\nabla u\| _{\dot{B}_{p,1}^{\frac{d}{p}}}.
\end{equation*}
It follows from \eqref{Eq:5.3} with $\sigma=\frac{d}{p}-1$ and Proposition \ref{Prop2.6} that
\begin{eqnarray*}
\|T_{\nabla u}\nabla K(a)\|^{\ell}_{\dot{B}^{\frac{d}{2}-1}_{2,1}}
\lesssim \|T_{\nabla u}\nabla K(a)\|^{\ell}_{\dot{B}^{\frac{d}{2}-2}_{2,1}}
\lesssim\|u\| _{\dot{B}_{p,1}^{\frac{d}{p}}}\|a\| _{\dot{B}_{p,1}^{\frac{d}{p}}}.
\end{eqnarray*}
Then we discover that
\begin{equation*}
\|g_{4}(a,u)\|^{\ell}_{\dot{B}_{2,1}^{\frac{d}{2}-1}}\lesssim \mathcal{E}_{p}(t)\big(\|(\tilde{a},u)^{\ell}\|_{\dot{B}^{\frac{d}{2}+1}_{2,1}}+\|a\|^{h}_{\dot{B}^{\frac{d}{p}}_{p,1}}
+\|u\|_{\dot{B}_{p,1}^{\frac {d}{p}+1}}^{h}\big).
\end{equation*}
Hence, the the inequality \eqref{Eq:5.2} is proved.

Inserting the above estimates into \eqref{Eq:5.1} and performing the fact that $\mathcal{E}_{p}(t)\lesssim \mathcal{E}_{p,0}\ll1$ for all $t\geq 0$ guaranteed by Theorem \ref{Thm1.1}, we end up with
\begin{equation*}
\frac{d}{dt}\big(\|(\tilde{a},u)^{\ell}\|_{\dot{B}^{\frac{d}{2}-1}_{2,1}}+\|(\nabla a,u)\|_{\dot{B}_{p,1}^{\frac {d}{p}-1}}^{h}\big)
+\big(\|(\tilde{a},u)^{\ell}\|_{\dot{B}^{\frac{d}{2}+1}_{2,1}}+\|\nabla a\|_{\dot{B}_{p,1}^{\frac {d}{p}-1}}^{h}+\|u\|_{\dot{B}_{p,1}^{\frac {d}{p}+1}}^{h}\big)\leq 0.
\end{equation*}

In what follows, we observe that interpolation plays the key role to obtain the time-decay estimates. Owing to the fact
$-s_{1}<\frac{d}{2}-1\leq \frac{d}{p}<\frac{d}{2}+1$, we get from Proposition \ref{Prop2.2} (the real interpolation) that
\begin{equation*}
\|(\tilde{a},u)^{\ell}\|_{\dot{B}^{\frac{d}{2}-1}_{2,1}}\lesssim \big(\|(\tilde{a},u)\|^{\ell}_{\dot{B}^{-s_{1}}_{2,\infty}}\big)^{\theta_{0}}
\big(\|(\tilde{a},u)\|^{\ell}_{\dot{B}^{\frac{d}{2}+1}_{2,\infty}}\big)^{1-\theta_{0}} \ \ \hbox{with} \ \ \theta_{0}=\frac{2}{d/2+1+s_{1}}\in(0,1).
\end{equation*}
With the aid of \eqref{Eq:4.12}, we arrive at
\begin{equation*}
\|(\tilde{a},u)\|^{\ell}_{\dot{B}^{\frac{d}{2}+1}_{2,\infty}}\geq c_{0}\big(\|(\tilde{a},u)^{\ell}\|_{\dot{B}^{\frac{d}{2}-1}_{2,1}}\big)^{\frac{1}{1-\theta_{0}}} \ \ \hbox{with} \ \ c_{0}=C^{-\frac{1}{1-\theta_{0}}}C_{0}^{-\frac{\theta_{0}}{1-\theta_{0}}}.
\end{equation*}
On the other hand, it follows from the fact $\|(\nabla a,u)\|^{h}_{\dot{B}^{\frac{d}{p}-1}_{p,1}}\leq \mathcal{E}_{p}(t)\lesssim \mathcal{E}_{p,0}\ll 1$ for all $t\geq0$ that
\begin{equation*}
\|\nabla a\|^{h}_{\dot{B}^{\frac{d}{p}-1}_{p,1}}\geq (\|\nabla a\|^{h}_{\dot{B}^{\frac{d}{p}-1}_{p,1}})^{\frac{1}{1-\theta_{0}}} \ \ \hbox{and} \ \ \|u\|^{h}_{\dot{B}^{\frac{d}{p}+1}_{p,1}}\geq (\|u\|^{h}_{\dot{B}^{\frac{d}{p}-1}_{p,1}})^{\frac{1}{1-\theta_{0}}}.
\end{equation*}
Consequently, there exists a constant $\tilde{c}_{0}>0$ such that the following Lyapunov-type inequality holds
\begin{equation}\label{Eq:5.5}
\frac{d}{dt}\big(\|(\tilde{a},u)^{\ell}\|_{\dot{B}^{\frac{d}{2}-1}_{2,1}}+\|(\nabla a,u)\|_{\dot{B}_{p,1}^{\frac {d}{p}-1}}^{h}\big) +\tilde{c}_{0}\big(\|(\tilde{a},u)^{\ell}\|_{\dot{B}^{\frac{d}{2}-1}_{2,1}}+\|(\nabla a,u)\|_{\dot{B}_{p,1}^{\frac {d}{p}-1}}^{h}\big) ^{1+\frac{2}{d/2-1+s_{1}}}\leq 0.
\end{equation}
Solving \eqref{Eq:5.5} directly yields
\begin{eqnarray}\label{Eq:5.6}
\|(\tilde{a},u)^{\ell}\|_{\dot{B}^{\frac{d}{2}-1}_{2,1}}+\|(\nabla a,u)\|_{\dot{B}_{p,1}^{\frac {d}{p}-1}}^{h}
&\leq&\big(\mathcal{E}_{p,0}^{-\frac{2}{d/2-1+s_{1}}}+\frac{2\tilde{c}_{0}t}{d/2-1+s_{1}}\big)^{-\frac{d/2-1+s_{1}}{2}} \nonumber\\
&\lesssim & (1+t)^{-\frac{d/2-1+s_{1}}{2}} \ \ \hbox{for all} \ \ t\geq 0.
\end{eqnarray}
It follows from Proposition \ref{Prop2.3} and \eqref{Eq:5.6} that
\begin{equation}\label{Eq:5.7}
\|a\|_{\dot{B}_{p,1}^{\frac{d}{p}-2}}\lesssim\|a^{\ell}\|_{\dot{B}_{2,1}^{\frac{d}{2}-2}}+\|a\|^{h}_{\dot{B}_{p,1}^{\frac{d}{p}-2}}
\lesssim \|\tilde{a}^{\ell}\|_{\dot{B}_{2,1}^{\frac{d}{2}-1}}+\|a\|^{h}_{\dot{B}_{p,1}^{\frac{d}{p}}}
\lesssim (1+t)^{-\frac{d}{4}-\frac{s_{1}-1}{2}}
\end{equation}
and thus that
\begin{equation}\label{Eq:5.8}
\|u\|_{\dot{B}_{p,1}^{\frac{d}{p}-1}}\lesssim\|u^{\ell}\|_{\dot{B}_{2,1}^{\frac{d}{2}-1}}+\|u\|^{h}_{\dot{B}_{p,1}^{\frac{d}{p}-1}}
\lesssim (1+t)^{-\frac{d}{4}-\frac{s_{1}-1}{2}} \ \ \hbox{for all} \ \ t\geq 0.
\end{equation}
On the other hand, let $\tilde{s}_{1}\triangleq s_{1}+d(\frac{1}{2}-\frac{1}{p})$. We get from Propositions \ref{Prop2.2} and \ref{Prop2.3} that
\begin{equation}\label{Eq:5.9}
\|a^{\ell}\|_{\dot{B}_{p,1}^{s}}\lesssim \|a^{\ell}\|_{\dot{B}_{2,1}^{s+d(\frac{1}{2}-\frac{1}{p})}}\lesssim
\big(\|\tilde{a}\|^{\ell}_{\dot{B}_{2,\infty}^{-s_{1}}}\big)^{\theta_{1}} \big(\|\tilde{a}\|^{\ell}_{\dot{B}_{2,\infty}^{\frac{d}{2}-1}}\big)^{1-\theta_{1}}
\end{equation}
if
$s\in \big(-\tilde{s}_{1}-1,\frac{d}{p}-2\big) \ \ \hbox{and} \ \ \theta_{1}=\frac{\frac{d}{p}-2-s}{\frac{d}{2}-1+s_{1}} \in (0,1)$,
and that
\begin{equation}\label{Eq:5.10}
\|u\|^{\ell}_{\dot{B}_{p,1}^{s}}\lesssim \|u\|^{\ell}_{\dot{B}_{2,1}^{s+d(\frac{1}{2}-\frac{1}{p})}}\lesssim
\big(\|u\|^{\ell}_{\dot{B}_{2,\infty}^{-s_{1}}}\big)^{\theta_{2}} \big(\|u\|^{\ell}_{\dot{B}_{2,\infty}^{\frac{d}{2}-1}}\big)^{1-\theta_{2}}
\end{equation}
if
$s\in \big(-\tilde{s}_{1},\frac{d}{p}-1\big) \ \ \hbox{and} \ \ \theta_{2}=\frac{\frac{d}{p}-1-s}{\frac{d}{2}-1+s_{1}} \in (0,1)$.
Noticing the fact that
\begin{equation*}
\|(\tilde{a},u)\|^{\ell}_{\dot{B}^{-s_{1}}_{2,\infty}}\leq C_{0}
\end{equation*}
for all $t\geq 0$, with the aid of \eqref{Eq:5.6}, \eqref{Eq:5.9}, \eqref{Eq:5.10}, we have
\begin{equation*}
\|a^{\ell}(t)\|_{\dot{B}_{p,1}^{s}}\lesssim \big[(1+t)^{-\frac{d/2-1+s_{1}}{2}}\big]^{1-\theta_{1}}=(1+t)^{-\frac{d}{2}(\frac{1}{2}-\frac{1}{p})-\frac{s_{1}+s+1}{2}}
\end{equation*}
and
\begin{equation*}
\|u^{\ell}(t)\|_{\dot{B}_{p,1}^{s}}\lesssim \big[(1+t)^{-\frac{d/2-1+s_{1}}{2}}\big]^{1-\theta_{2}}=(1+t)^{-\frac{d}{2}(\frac{1}{2}-\frac{1}{p})-\frac{s_{1}+s}{2}}
\end{equation*}
for all $t\geq 0$, which lead to
\begin{equation} \label{Eq:5.11}
\|a(t)\|_{\dot{B}_{p,1}^{s}}\lesssim\|a^{\ell}(t)\|_{\dot{B}_{p,1}^{s}}+\|a(t)\|^{h}_{\dot{B}_{p,1}^{s}}
\lesssim (1+t)^{-\frac{d}{2}(\frac{1}{2}-\frac{1}{p})-\frac{s_{1}+s+1}{2}}
\end{equation}
for $s\in(-\tilde{s}_{1}-1,\frac{d}{p}-2)$, and
\begin{equation}\label{Eq:5.12}
\|u(t)\|_{\dot{B}_{p,1}^{s}}\lesssim\|u^{\ell}(t)\|_{\dot{B}_{p,1}^{s}}+\|u(t)\|^{h}_{\dot{B}_{p,1}^{s}}\lesssim (1+t)^{-\frac{d}{2}(\frac{1}{2}-\frac{1}{p})-\frac{s_{1}+s}{2}}
\end{equation}
for $s\in(-\tilde{s}_{1},\frac{d}{p}-1)$.
Hence, combining with \eqref{Eq:5.7}, \eqref{Eq:5.8}, \eqref{Eq:5.11}, \eqref{Eq:5.12} yields Theorem \ref{Thm1.2}.
\subsection{Proof of Corollary \ref{Cor1.1}}
For brevity, it suffices to establish those time-optimal decay rates of $a$. Thanks to $p\leq r\leq\infty$, it follows from the embeddings $\dot{B}^{0}_{r,1}\hookrightarrow L^{r}$ and $\dot{B}_{p,1}^{l+d\big(\frac{1}{p}-\frac{1}{r}\big)}\hookrightarrow \dot{B}_{r,1}^{l}$ $(l\in\mathbb{R})$, \eqref{Eq:5.7} and \eqref{Eq:5.11} that
\begin{equation*}
\|\Lambda ^{l}a\|_{L^{r}}\lesssim \|a\|_{\dot{B}_{r,1}^{l}}\lesssim \|a\|_{\dot{B}_{p,1}^{l+d\big(\frac{1}{p}-\frac{1}{r}\big)}}
\lesssim
(1+t)^{-\frac{d}{2}(\frac{1}{2}-\frac{1}{r})-\frac{s_{1}+l+1}{2}}
\end{equation*}
for $p\leq r\leq\infty$ and $l\in\mathbb{R}$ satisfying $-\tilde{s}_{1}-1<l+d\big(\frac{1}{p}-\frac{1}{r}\big)\leq\frac{d}{p}-2$. This completes the proof of Corollary \ref{Cor1.1}.

\section*{Acknowledgments}
Last but not least, he is very grateful to Professor J. Xu for the suggestion on this question.

\end{document}